\documentclass[12pt]{amsart}

\usepackage[a4paper, margin=3cm]{geometry}
\usepackage{amsmath,amssymb}
\usepackage{mathtools}
\usepackage[capitalize]{cleveref}
\usepackage{todonotes}
\usepackage{tikz-cd}
\tikzcdset{
  cells={font=\everymath\expandafter{\the\everymath\displaystyle}},
}
\usepackage{url}
\def\DefineSymbol#1#2{\newcommand{#1}{{\mathrm {#2}}}}
\def\DefineCategory#1#2{\newcommand{#1}{{\mathbf {#2}}}}
\theoremstyle{plain}
    \newtheorem{theorem}{Theorem}[section]
    \newtheorem{lemma}[theorem]{Lemma}
    \newtheorem{proposition}[theorem]{Proposition}
    
    \newtheorem{fact}[theorem]{Fact}
\theoremstyle{definition}
    \newtheorem{definition}[theorem]{Definition}
    \newtheorem{construction}[theorem]{Construction}
\theoremstyle{remark}
    \newtheorem{remark}[theorem]{Remark}
    \newtheorem{example}[theorem]{Example}
    \numberwithin{equation}{section}
\DefineSymbol{\id}{id}
\DefineSymbol{\proj}{proj}
\DefineSymbol{\incl}{incl}
\DefineSymbol{\const}{const}
\DefineSymbol{\op}{op}
\DefineSymbol{\res}{res}
\DeclareMathOperator{\rank}{rank}

\DeclareMathOperator{\ind}{ind}

\DeclareMathOperator{\Hom}{Hom}
\DeclareMathOperator{\End}{End}
\DeclareMathOperator{\GL}{GL}

\DeclareMathOperator{\Coker}{Coker}
\let\Im\relax
\DeclareMathOperator{\Im}{Im}

\DeclareMathOperator{\Herm}{Herm}
\let\sf\relax
\DeclareMathOperator{\sf}{sf}
\DefineCategory{\Set}{Set}
\DefineCategory{\Ab}{Ab}
\DefineCategory{\Mod}{Mod}
\DefineCategory{\Alg}{Alg}
\newcommand{\Abb}{\mathbb{A}}

\newcommand{\Cbb}{\mathbb{C}}

\newcommand{\Rbb}{\mathbb{R}}

\newcommand{\Zbb}{\mathbb{Z}}

\newcommand{\Ecal}{\mathcal{E}}
\newcommand{\Fcal}{\mathcal{F}}

\newcommand{\Kcal}{\mathcal{K}}

\newcommand{\Mcal}{\mathcal{M}}

\newcommand{\Ocal}{\mathcal{O}}

\newcommand{\Scal}{\mathcal{S}}
\newcommand{\Tcal}{\mathcal{T}}

\newcommand{\Zcal}{\mathcal{Z}}

\newcommand{\Qfrak}{\mathfrak{Q}}

\renewcommand{\tilde}{\widetilde}
\renewcommand{\hat}{\widehat}
\renewcommand{\bar}{\overline}

\newcommand{\setmid}{\mathrel{}\middle|\mathrel{}}
\newcommand{\rest}[2]{\left.#1\right|_{#2}}
\newcommand{\abs}[1]{\lvert #1 \rvert}
\newcommand{\norm}[1]{\lVert #1 \rVert}

\newcommand{\Spec}[0]{\mathrm{Spec}}
\newcommand{\Conf}[0]{\mathrm{Conf}}
\newcommand{\conf}[0]{\mathrm{config}}
\newcommand{\Poly}[0]{\mathrm{Poly}}
\newcommand{\disk}[0]{{B^2}}
\newcommand{\mobius}[0]{\Phi}
\newcommand{\bundle}{\mathrm{Bun}}

\begin{document}

\begin{abstract}
    In this paper, we give alternative proofs of Bott periodicity of $ K $-theory and the bulk-edge correspondence of integer quantum Hall effect. 
    Regarding Bott periodicity, we connect its proof with configuration spaces and use Quot schemes in algebraic geometry in our proof. 
    Regarding the bulk-edge correspondence, we formulate edge indices based on the consideration of Graf--Porta and give a more elementary and self-contained proof. 
\end{abstract}

\title[Bott periodicity via Quot schemes and bulk-edge correspondence]{A proof of Bott periodicity via Quot schemes and bulk-edge correspondence}
\author{Masaki Natori}
\date{}
\address{Graduate School of Mathematical Sciences, University of Tokyo, 3-8-1 Komaba,
Tokyo, 153-8914, Japan.}
\email{{\tt natori-masaki616@g.ecc.u-tokyo.ac.jp}}
\maketitle

\setcounter{tocdepth}{1}
\tableofcontents

\section{Introduction}\label{sec:intro}

\subsection{Overview}
In this paper, we give alternative proofs of Bott periodicity of $ K $-theory and the bulk-edge correspondence of integer quantum Hall effect. 

A feature of our proof of the Bott periodicity theorem (\cref{sec:strategy}, \ref{sec:construction}, \ref{sec:continuity} and \ref{sec:proof-of-bott-periodicity}) is that we use spaces of configurations labeled by vector spaces on $ B^2 \coloneqq \{\abs{z} < 1\} $. 
G. B. Segal reveals that connective $ K $-homology theory is represented by spaces of configurations labeled by vector spaces \cite{MR0515311}. 
Our proof is based on this point of view. 
Furthermore, this view gives an idea of our proof of the bulk-edge correspondence (\cref{sec:setting} and \ref{sec:proof-of-bulkedge}).
The bulk indices are related to the construction of the inverse map of the Bott map due to M. F. Atiyah and R. Bott \cite{MR0178470}. 
On the other hand, the edge indices are related to the one using family indices due to Atiyah \cite{MR0228000}. 
Since our proof of Bott periodicity is related to both of the constructions, we can rewrite the bulk and edge indices in our view.
As a result, we obtain an elementary and self-contained proof of the bulk-edge correspondence. 

To prove Bott periodicity, it is enough to show that the Bott map $ \beta_X \colon K(X) \to K(X \times (D^2, S^1)) $ has its inverse map for each compact Hausdorff spaces $ X $. 
To do so, we need to construct vector bundles over $ X $ from elements in $ K(X \times (D^2, S^1)) $, and as a preliminary step we construct families of configurations labeled by vector spaces on $ B^2 $.
In concrete, we will show the following property of the configuration map $ \conf_{\disk} $:
\begin{theorem}[see \cref{thm:coker-conti2}]
    The configuration map 
    \[
        \conf_{\disk} \colon \Poly_{V, S^1}^{l, r} \to \Conf_{V, \disk}^r
    \]
    is continuous. 
\end{theorem}
\noindent
Here, $ \Poly_{V, S^1}^{l, r} $ is a space of matrix-coefficient polynomials that are invertible over $ S^1 $. 
It will be defined in \cref{sec:continuity}. 
Also, $ \Conf_{V, \disk}^r $ is a configuration space labeled by vector spaces over $ \disk $. 
It will be defined in \cref{sec:construction}. 
This map roughly corresponds to the homotopy inverse $ \Omega U \to \Zbb \times BU $ of the Bott map $ \Zbb \times BU \to \Omega U $.
To show this continuity, we take a procedure of constructing a morphism of schemes using Quot schemes in algebraic geometry and obtaining the configuration map as its analytification.

\subsection{On the Bott periodicity theorem}

Various methods to prove Bott periodicity are known.
It was first established by Bott in 1959 using Morse theory \cite{MR0110104}. 
Subsequently, an elementary proof was given by Atiyah--Bott \cite{MR0178470}, and a proof using family indices was given by Atiyah \cite{MR0228000}.
There is also a proof using Kuiper's theorem by Atiyah and I. M. Singer \cite{MR0285033}, a proof using quasifibration by D. MacDuff \cite{MR0467734}, a concise proof using spectral decomposition of unitary matrices and group completion theorem by B. Harris \cite{MR0563240}, and so on. 

We compare our proof with the proof by Atiyah--Bott \cite{MR0178470} and the proof by Atiyah \cite{MR0228000}. 
In the construction by Atiyah--Bott, after they approximating a clutching function of a vector bundle over $ S^2 $ by a Laurent polynomial, it is transformed into a polynomial, then into a linear polynomial, and finally into a vector bundle by taking the Bloch bundle.
In contrast, our proof has a feature that we directly obtain a vector bundle from the polynomial. 
Restricting the case of linear polynomials, two constructions are equivalent. 
The construction by Atiyah takes the Toeplitz operator of a clutching function and obtains an element of $ K $-group as its family index. 
In \cite{MR0228000}, Atiyah compares Atiyah--Bott's proof with Atiyah's proof. 
There, a construction in which the polynomial is regarded as a $ \Cbb[z] $-homomorphism and its cokernel is taken is mentioned as a construction that is intermediate between the two constructions.
We reconsider the construction as sheaf homomorphisms, and as mentioned in the beginning, we clarify that there is a step to obtain families of configurations labeled by vector spaces from polynomials as a preliminary step of the operation to obtain vector bundles from polynomials.

We describe the idea of our proof of Bott periodicity in three stages.
The first and second ideas are not actually implemented.
The third idea is the strategy used in the actual proof.
We explain the first idea. 
For the proof, we only need to construct the homotopy inverse $ \alpha \colon \Omega U \to \Zbb \times BU $ of the Bott map $ \beta \colon \Zbb \times BU \to \Omega U $.
By spectral decomposition, an element of $ U(n) $ can be regarded as a configuration on $ S^1 $ labeled by vector spaces. 
Thus, an element of $ \Omega U $ is an $ S^1 $-family of configurations on $ S^1 $ labeled by vector spaces. 
Roughly speaking, the map $ \alpha \colon \Omega U \to \Zbb \times BU $ can be considered as a map returning signed vector spaces obtained as its "winding number". 
This is the first idea.
In the second idea, we define this winding number in a specific case. 
That is, let $ V $ be a finite-dimensional complex vector space and consider the case where $ f(z) \colon S^1 \to \GL_\Cbb(V) $ is given in the form of a polynomial
\[
    f(z) = \sum_{j = 0}^l a_j z^j \quad \text{($ \forall j $, $ a_j \in \End_\Cbb(V) $)}.
\] 
Then, by generalized eigendecomposition of $ f(z) \in \End_\Cbb(V) $ obtained for each $ z \in \Cbb $, we get a spectral curve on $ \Cbb \times \Cbb $ labeled by generalized eigenspaces. 
Restricting this to $ \disk \times \{0\} $, we obtain an configuration on $ \disk $.
When restricting, we will count the vector spaces of labels as many times as the order at each zero. 
Finally, taking the direct sum of all labels on the configuration, we obtain the desired vector space. 
This is the second idea. 
There is a possibility that this is rigorously formulated using the moduli stack of coherent sheaves on $ \Abb_{\Cbb}^2 $. 
However, we will not take that strategy this time.
Instead, as the third idea, to define configurations on $ \disk $ obtained as a restriction to $ \disk \times \{0\} $ rigorously, we will consider an operation of taking cokernels of sheaf homomorphisms. 
In this paper, we will do this for families and construct a configuration map $ \conf_{\disk} \colon \Poly_{V, S^1}^{l, r} \to \Conf_{V, \disk}^r $ which obtain configurations on $ \disk $ from matrix-coefficient polynomials $ f(z) $.

\subsection{On the bulk-edge correspondence}

The integer quantum Hall effect was first experimentally discovered by K. v. Klitzing, G. Dorda and M. Pepper in 1980 \cite{PhysRevLett.45.494}. 
The bulk indices was introduced by D. J. Thouless, M. Kohmoto, M. P. Nightingale and M. den Nijs as a first Chern number of a Bloch bundle (called TKNN number) \cite{PhysRevLett.49.405}. 
The bulk-edge correspondence of integer quantum Hall effect was first shown by Y. Hatsugai in 1993 using Riemann surfaces \cite{PhysRevLett.71.3697}. 
Later, a mathematical proof was given by J. Kellendonk, T. Richter and H. Schulz-Baldes using $ K $-theory of $ C^* $-algebras \cite{MR1877916}. 
G. M. Graf and M. Porta introduced some different vector bundles to represent edge indices \cite{MR3123539}. 
In other words, they constructed vector bundles using solutions of Hamiltonians decaying in a particular direction and associated their indices with edge indices.
S. Hayashi considered Graf--Porta's vector bundles in terms of $ K $-theory and index theory \cite{MR3720513}. 
Furthermore, Hayashi showed the coincidence between the indices and the edge indices and the coincidence between the indices and the bulk indices using the cobordism invariance of indices and finally proved the bulk-edge correspondence.

In contrast, in this paper, the indices of the Graf--Porta vector bundle are formulated as the family indices of Fredholm operators.
In this paper, we will call that index the edge index.
The coincidence between the edge index and the bulk index is then shown in a more elementary and self-contained manner.
Using this formulation, we also define edge indices for Hamiltonians that are not necessarily self-adjoint and prove the bulk-edge correspondence.

We explain the idea of the proof of the bulk-edge correspondence.
Let $ X $ be an oriented smooth $ n $-dimensional closed manifold, $ \gamma $ a simple closed curve in complex plane $ \Cbb $ and $ V $ a finite-dimensional complex vector space. 
Let $ D_{\gamma} $ be the bounded one of the two connected components of $ \Cbb \setminus \gamma $.
Suppose a continuous map
\[
    H \colon X \times S^1 \times \gamma \to \GL_\Cbb(V)
\]
is given in the form
\[
    H(x, z, \lambda) = \sum_{j = -k}^l a_j(x) z^j - \lambda. 
\]
Here, for each $ 0 \leq j \leq l $, $ a_j \colon X \to \End_\Cbb(V) $ is a continuous map. 
We can regard $ H $ as a multivalued function labeled by vector spaces from $ X \times S^1 \times \bar{D_\gamma} $ to $ \Cbb $ by generalized eigendecomposition. 
The bulk indices are indices obtained from the intersection between its graph and $ X \times S^1 \times D_\gamma \times \{0\} $. 
On the other hand, when $ z^k H $ is regarded as a multivalued function labeled by vector spaces from $ X \times D^2 \times \gamma $ to $ \Cbb $ in the same way, the edge indices are indices obtained from the intersection between its graph and $ X \times \disk \times \gamma \times \{0\} $. 
Gluing these multivalued functions, we regard it as a multivalued function from $ X \times S^1 \times \bar{D_\gamma} \cup_{X \times S^1 \times \gamma} X \times D^2 \times \gamma \approx X \times S^3 $ to $ \Cbb $. 
Then, the index obtained from the intersection between its graph and $ X \times S^3 \times \{0\} $ is expected to be 0 since $ X \times S^3 $ is a closed manifold. 
This is an idea of the proof of the bulk-edge correspondence. 
We will describe a low-dimensional version of this idea in \cref{sec:toymodel}. 
Furthermore, although multivalued functions do not appear explicitly, the proof of the bulk-edge correspondence based on this idea is given in \cref{sec:setting} and \ref{sec:proof-of-bulkedge}.

Our proof of Bott periodicity can be extended to the case of $ KR $-theory \cite{MR0206940} and the bulk-edge correspondence can also be extended to the case of Real spaces. 

This paper is divided into two parts: 
a part giving an alternative proof of Bott periodicity (\cref{part:bott-periodicity}) and a part giving an alternative proof of the bulk-edge correspondence (\cref{part:bulkedge}).
\cref{part:bott-periodicity} does not depend on \cref{part:bulkedge}. 
The only place where \cref{part:bulkedge} depends on \cref{part:bott-periodicity} is in citing \cref{sec:relation}.

\subsection*{Acknowledgment}

I am deeply grateful to my supervisor Mikio Furuta for his enormous support and helpful advice. 
He checked the draft and gave me useful comments. 
I would also like to thank Tasuki Kinjo and Yugo Takanashi. 
They gave me a lot of helpful comments in algebraic geometry. 
I would also like to thank Nobuo Iida and Jin Miyazawa for checking the draft and continuous encourage. 
This work was supported by JSPS KAKENHI Grant Number 23KJ0402 and Forefront Physics and Mathematics Program to Drive Transformation.

\part{An alternative proof of the Bott periodicity theorem}\label{part:bott-periodicity}

\section{The statement and strategy of the proof}\label{sec:strategy}

In this section, we will describe the statement of the Bott periodicity theorem and the strategy of the proof in \cref{sec:construction}, \ref{sec:continuity} and \ref{sec:proof-of-bott-periodicity}. 
At first, we will review the definition of $ K $-group. 

\begin{definition}
    Let $ X $ be a compact Hausdorff space and $ A $ a closed subspace of $ X $. 
    Let $ \Kcal(X, A) $ denote the abelian monoid of isomorphism classes of triples $ (E, F, \alpha) $ which consists of vector bundles $ E $, $ F $ over $ X $ and a bundle isomorphism $ \alpha \colon \rest{E}{A} \to \rest{F}{A} $. 
    Then the relative $ K $-group $ K(X, A) $ of $ (X, A) $ is defined by 
    \[
        K(X, A) \coloneqq \Kcal(X, A) / {\sim}. 
    \]
    Here $ \sim $ stands for an equivalent relation on $ \Kcal(X, A) $ generated by following relations: 
    \begin{enumerate}
        \item $ (E, F, \alpha) \sim (E, F, \alpha') $ when $ \alpha $ and $ \alpha' $ are homotopic, preserving that they are isomorphic on $ A $ and
        \item $ (E, F, \alpha) \sim (E \oplus G, F \oplus G, \alpha \oplus \id) $ for any vector bundles $ G $ over $ X $. 
    \end{enumerate}
    We write $ [E, F, \alpha] $ for the element of $ K(X, A) $ which corresponds to a triple $ (E, F, \alpha) $. 
    Also, the $ K $-group of $ X $ is defined by $ K(X) \coloneqq K(X, \emptyset) $. 
    For a vector bundle $ E $ over $ X $, we write $ [E] \coloneqq [E, 0, \emptyset] \in K(X) $. 
\end{definition}

Then, the statement of the Bott periodicity theorem is as follows. 
We write $ D^2 \coloneqq \{z \in \Cbb \mid \abs{z} \leq 1\} $, $ \disk \coloneqq D^2 \setminus \partial D^2 $, $ S^1 \coloneqq \partial D^2 $. 

\begin{theorem}[\cite{MR0178470}, \cite{MR0228000}]\label{thm:bott-periodicity}
    Let $ X $ be a compact Hausdorff space. 
    Then, the Bott map
    \begin{align*}
        \beta_X \colon K(X) & \to K(X \times (D^2, S^1)) ; \\
        [E] & \mapsto [E] \boxtimes [\underline{\Cbb}, \underline{\Cbb}, z^{-1}] = [\proj_X^*E, \proj_X^*E, z^{-1}]
    \end{align*}
    is an isomorphism. 
    Here $ z $ is a coordinate of $ S^1 $, $ [\underline{\Cbb}, \underline{\Cbb}, z^{-1}] $ is an element of $ K(D^2, S^1) $, $ \boxtimes $ stands for the external tensor product. 
\end{theorem}

Any element of $ K(X \times (D^2, S^1)) $ can be written by $ [\proj_X^*E, \proj_X^*E, H] $ using a vector bundle $ E $ over $ X $ and a continuous family of Laurent polynomial $ \{H_x(z) = \sum_{j = - k}^l a_j(x) z^j\}_{x \in X} $ ($ \forall x \in X $, $ a_j(x) \in \End_{\Cbb}(E_x) $) (see \cref{sec:proof-of-bott-periodicity}). 
In this section, for simplicity, we will construct a vector bundle $ \bundle_X(V, H) $ over $ X $ which eventually satisfies $ \beta_X^{-1}([\underline{V}, \underline{V}, H]) = -[\bundle_X(V, H)] $ in the case that the vector bundle $ E $ is trivial and k = 0, namely the Laurent polynomial is a polynomial by using the results in the following sections. 

\begin{construction}\label{const:basic}
    Once again, the setup is described as follows. 
    Let $ V $ be a finite-dimensional vector space and $ l \geq 0 $. 
    Suppose for each $ 0 \leq j \leq l $, a continuous map $ a_j \colon X \to \End_{\Cbb}(V) $ is given. 
    Then, we obtain a continuous family of $ \End_{\Cbb}(V) $-coefficient polynomial
    \[
        \left\{H_x(z) \coloneqq \sum_{j = 0}^l a_j(x) z^j\right\}_{x \in X}. 
    \]
    For $ x \in X $, $ z \in S^1 $, we assume that $ H_x(z) \in \GL_\Cbb(V) $. 
    Also, we assume that the number of roots with multiplicity in $ \disk $ of $ \det(H_x) $ is constant $ r $, independent of $ x \in X $.

    The family of polynomials $ \{H_x(z)\}_{x \in X} $ defines a continuous map $ H \colon X \to \Poly_{V}^l $ using the notation in \cref{def:poly}. 
    By assumption, the image of $ H $ is included in $ \Poly_{V, S^1}^{l, r} $. 
    The map $ \conf_{\disk} \colon \Poly_{V, S^1}^{l, r} \to \Conf_{V, \disk}^r $ is continuous by \cref{thm:coker-conti2} and there exists a vector bundle $ \Ecal_{V, \disk}^r \to \Conf_{V, \disk}^r $ over $ \Conf_{V, \disk}^r $ by \cref{prop:global-section-bundle}. 
    Based on the above preparation, we obtain a vector bundle $ \bundle_X(V, H) \coloneqq \conf(H)^* \Ecal_{V, \disk}^r \to X $. 
    Here, we define $ \conf(H) \coloneqq \conf_{\disk} \circ H \colon X \to \Conf_{V, \disk}^r $. 

    \[
        \begin{tikzcd}[column sep=large]
            \bundle_X(V, H) \arrow{rr} \arrow{d} \arrow[dr, phantom, very near start, "\lrcorner"] &  & \Ecal_{V, \disk}^r \arrow{d} \\
            X \arrow{r}[swap]{H} & \Poly_{V, S^1}^{l, r} \arrow{r}[swap]{\conf_{\disk}} & \Conf_{V, \disk}^r
        \end{tikzcd}
    \]
\end{construction}

\subsection*{Structure of \cref{part:bott-periodicity}}
In \cref{sec:construction}, we will construct the configuration spaces $ \Conf_{V, \disk}^r $ and the global section bundles $ \mathcal{E}_{V, \disk}^r \to \Conf_{V, \disk}^r $ and compare with Quot schemes. 
In \cref{sec:continuity}, we will define the configuration maps $ \conf_{\disk} \colon \Poly_{V, S^1}^{l, r} \to \Conf_{V, \disk}^r $ and prove its continuity. 
This configuration map, which sends matrix-coefficient polynomials to configurations on $ \disk $ is a key map in our proof of the Bott periodicity theorem. 
In \cref{sec:proof-of-bott-periodicity}, based on the above, we will define a homomorphism $ \alpha_X \colon K(X \times (D^2, S^1)) \to K(X) $ by using $ \bundle_X(V, H) $ in this section,  confirm well-definedness and show that it is a inverse map of the Bott map $ \beta_X $. 
In \cref{sec:relation}, we will compare our proof with Atiyah--Bott's proof and Atiyah's proof. 

\section{Construction of configuration spaces}\label{sec:construction}

In this section, we will construct the configuration space $ \Conf_{V, U}^r $ and the global section bundle $ \mathcal{E}_{V, U}^r \to \Conf_{V, U}^r $ and compare with Quot scheme. 

\subsection{Construction of configuration spaces and global section bundles}\label{subsec:construction}

\begin{definition}
    Let $ V $ be a finite-dimensional complex vector space, $ U $ an open set in complex plane $ \Cbb $ and $ r $ a non-negative integer. 
    A pair $ (A, \iota) $ is called $ V $-configuration of rank $ r $ over $ U $ if a matrix $ A \in M_r(\Cbb) $ and a linear map $ \iota \colon V \to \Cbb^r $ satisfy the following conditions:
    \begin{itemize}
        \item The set of eigen values $ \sigma(A) $ of $ A $ is included in $ U $ and
        \item $ \displaystyle \sum_{i = 0}^{r - 1} A^i \iota \colon V^{\oplus r} \to \Cbb^r $ is surjective. 
    \end{itemize}
    We write $ \widetilde{\Conf}_{V, U}^r $ for the set of $ V $-configurations of rank $ r $ over $ U $.
    The space $ \widetilde{\Conf}_{V, U}^r $ has a topology as a subspace of $ M_r(\Cbb) \times \Hom_{\Cbb}(V, \Cbb^r) $. 
    A left action of $ \GL_r(\Cbb) $ on $ \widetilde{\Conf}_{V, U}^r $ is defined by conjugation to $ M_r(\Cbb) $ and post-composition to $ \Hom_{\Cbb}(V, \Cbb^r) $. 
\end{definition}

\begin{proposition}\label{prop:global-section-bundle}
    The left action $ \GL_r(\Cbb) \curvearrowright \widetilde{\Conf}_{V, U}^r $is free and proper. 
    In paticular, quotient map $ \widetilde{\Conf}_{V, U}^r \to \GL_r(\Cbb) \backslash \widetilde{\Conf}_{V, U}^r $ is a principal $ \GL_r(\Cbb) $ bundle,  and the quotient space$ \Conf_{V, U}^r \coloneqq \GL_r(\Cbb) \backslash \widetilde{\Conf}_{V, U}^r $ is a $ r \cdot \dim_{\Cbb}{V} $-dimensional complex manifold.
    Furthermore, we get a vector bundle $ \Ecal_{V, U}^r \coloneqq \widetilde{\Conf}_{V, U}^r \times_{\GL_r(\Cbb)} \Cbb^r \to \Conf_{V, U}^r $ of rank $ r $ as an associated bundle. 
\end{proposition}

We call $ \Conf_{V, U}^r $ $ V $- configuration space of rank $ r $ over $ U $, $ \Ecal_{V, U}^r \to \Conf_{V, U}^r $ the global section bundle over $ \Conf_{V, U}^r $. 
We write $ [A, \iota] \in \Conf_{V, U}^r $ for the equivalence class of $ (A, \iota) \in \tilde{\Conf}_{V, U}^r $. 

\begin{proof}
    Let $ \operatorname{Surj}(V^{\oplus r}, \Cbb^{r}) $ denote the $ \Cbb $-linear space of surjective $ \Cbb $-linear maps from $ V^{\oplus r} $ to $ \Cbb^r $. 
    The left action $ \GL_{r}(\Cbb) \curvearrowright \operatorname{Surj}(V^{\oplus r}, \Cbb^r) $ is free and proper, and the quotient space is the grassmannian. 
    Let $ \varphi \colon \widetilde{\Conf}_{V, U}^r \to \operatorname{Surj}(V^{\oplus r}, \Cbb^r) $ be defined by $ \varphi(A, \iota) \coloneqq \sum_{i = 0}^{r - 1} A^r \iota $, then $ \varphi $ is a $ \GL_r(\Cbb) $-equivariant map. 
    Thus, the action $ \GL_r(\Cbb) \curvearrowright \widetilde{\Conf}_{V, U}^r $ is also free and proper. 
\end{proof}

Bundle homomorphisms 
\begin{align*}
    A_{V, U}^r \colon \mathcal{E}_{V, U}^r \to \mathcal{E}_{V, U}^r, \quad
    \iota_{V, U}^r \colon \underline{V} \to \mathcal{E}_{V, U}^r
\end{align*}
over $ \Conf_{V, U}^r $ is defined as follows:
    First, let $ A_{V, U}^r \colon \mathcal{E}_{V, U}^r = \tilde{\Conf}_{V, U}^r \times_{\GL_r(\Cbb)} \Cbb^r \to \mathcal{E}_{V, U}^r = \tilde{\Conf}_{V, U}^r \times_{\GL_r(\Cbb)} \Cbb^r $ be defined by 
    \[
        A_{U, V}^r([A, \iota, v]) \coloneqq [A, \iota, Av]
    \]
    for $ (A, \iota, v) \in \tilde{\Conf}_{V, U}^r \times \Cbb^r $. 
    This map is well-defined. 
    Next, let $ \iota_{U, V}^r \colon \underline{V} = \Conf_{V, U}^r \times V \to \mathcal{E}_{V, U}^r = \tilde{\Conf}_{V, U}^r \times_{\GL_r(\Cbb)} \Cbb^r $ be defined by 
    \[
        \iota_{U, V}^r([A, \iota], w) \coloneqq [A, \iota, \iota w]
    \]
    for $ ([A, \iota], w) \in \Conf_{V, U}^r \times V $. 
    This map is also well-defined. 

\begin{lemma}\label{lem:configuration-space-one-to-one}
    For a topological space $ X $, there is a one-to-one correspondence as sets:
    \begin{align*}
        & \{\text{continuous maps }X \to \Conf_{V, U}^r\} \\
        & \cong \left\{(E, A \colon E \to E, \iota \colon \underline{V} \to E) \setmid
        \begin{aligned}
            & \text{$ E $ is a vector bundle of rank r over $ X $, }\\
            & \text{$ \sigma(A_x) \subset U $ for all $ x \in X $ and} \\
            & \text{$ \sum_{i = 0}^{r - 1} A^i \iota \colon \underline{V}^{\oplus r} \twoheadrightarrow E $ is a surjection. }
        \end{aligned}\right\} \Big/{\cong}. 
    \end{align*}
\end{lemma}

\begin{proof}
    Let us denote by $ \Scal $ the left-hand side set and $ \Tcal $ the right-hand side set. 
    First, a map $ F \colon \mathcal{S} \to \mathcal{T} $ is defined as follows:
    Take $ f \colon X \to \Conf_{V, U}^r \in \mathcal{S} $.
    Taking the pullback along $ f \colon X \to \Conf_{V, U}^r $ of the vector bundle $ \mathcal{E}_{V, U}^r \to \Conf_{V, U}^r $ and bundle homomorphisms $ A_{V, U}^r \colon \mathcal{E}_{V, U}^r \to \mathcal{E}_{V, U}^r $ and $ \iota_{V, U}^r \colon \underline{V} \to \mathcal{E}_{V, U}^r $, we obtain a vector bundle $ f^* \mathcal{E}_{V, U}^r \to X $ and bundle homomorphisms $ f^* A_{V, U}^r \colon f^* \mathcal{E}_{V, U}^r \to f^* \mathcal{E}_{V, U}^r $ and $ f^* \iota_{V, U}^r \colon \underline{V} \to f^* \mathcal{E}_{V, U}^r $. 
    Then, we define 
    \[
        F(f) \coloneqq [f^* \mathcal{E}_{V, U}^r, f^* A_{V, U}^r, f^* \iota_{V, U}^r] \in \mathcal{T}. 
    \]

    Next, a map $ G \colon \mathcal{T} \to \mathcal{S} $ is defined as follows:
    Take $ [E, A \colon E \to E, \iota \colon \underline{V} \to E] \in \mathcal{T} $. 
    Let $ \{O_\alpha\}_\alpha $ be a trivialization covering of the vector bundle $ E \to X $. 
    Then, for any $ \alpha $, we get a continuous map $ f_\alpha \colon O_\alpha \to \Conf_{V, U}^r $ from $ \rest{A}{O_\alpha} \colon \rest{E}{O_\alpha} \to \rest{E}{O_\alpha} $, $ \rest{\iota}{O_\alpha} \colon \underline{V} \to \rest{E}{O_\alpha} $ by using trivialization. 
    By construction, $ \{f_\alpha\}_\alpha $ are glued together and a continuous map $ f_E \colon X \to \Conf_{V, U}^r $ is defined.
    Then, we define $ G([E, A, \iota]) = f_E \in \mathcal{S} $. 

    By construction, $ F \colon \Scal \to \Tcal $ and $ G \colon \Tcal \to \Scal $ give the inverse each other. 
\end{proof}

Next, we will define configuration spaces over $ \Cbb P^1 = \Cbb \sqcup \{\infty\} $ and its open sets. 
Roughly speaking, by identifying $ V $-configuration spaces of rank $ r $ over $ \Cbb P^1 \setminus \{p\} $ with $ \Conf_{V, \Cbb}^r $ by M\"obius transformations and gluing them all over $ p \in \Cbb P^1 $, we get the $ V $-configuration space of rank $ r $ over $ \Cbb P^1 $. 
Specifically, it is constructed as follows:
Let $ \mobius_p $ be defined by 
\[
    \mobius_p(z) \coloneqq \frac{1}{z - p}
\]
for $ p \in \Cbb $. 
Let $ U $ be an open set of $ \Cbb $. 
For each $ p \in \Cbb P^1 $, prepare a copy of $ \Conf_{V, U}^r $ and write $ {}^p\Conf_{V, U}^r \coloneqq \Conf_{V, U}^r $. 
In the case of $ p \in \Cbb $, 
The element of $ {}^p \Conf_{V, U}^r $ corresponding to $ [A, \iota] \in \Conf_{V, U}^r $ is formally written as $ [\mobius_p^{-1}(A), \iota] $. 
Let $ \sim $ be an equivalence relation over $ {}^\infty\Conf_{V, \Cbb}^r \sqcup \coprod_{p \in \Cbb} {}^p\Conf_{V, \Cbb}^r $ generated by the following relations. 

\begin{enumerate}
    \item Let $ p \in \Cbb $. 
    Let $ [A, \iota] \in {}^\infty\Conf_{V, \Cbb}^r $ and suppose that $ p $ is not an eigenvalue of $ A $. 
    Then, 
    \[
        {}^\infty\Conf_{V, \Cbb}^r \ni [A, \iota] \sim [\mobius_p^{-1}(\mobius_p(A)), \iota]_p \in {}^p\Conf_{V, \Cbb}^r. 
    \]
    where $ \mobius_p(A) $ is defined by functional calculus. 
    We can show that $ [\mobius_p(A), \iota] \in \Conf_{V, \Cbb}^r $ by applying the following \cref{lem:relation} to the case that $ \mobius = \mobius_p $. 
    \item Let $ p, q \in \Cbb $, $ p \neq q $. 
    Let $ [A, \iota] \in {}^p\Conf_{V, \Cbb}^r $ and suppose that $ (q - p)^{-1} $ is not an eigenvalue of $ A $. 
    Then,  
    \[
        {}^p\Conf_{V, \Cbb}^r \ni [\mobius_p^{-1}(A), \iota]_p \sim [\mobius_q^{-1}(\mobius_q\mobius_p^{-1}(A)), \iota]_q \in {}^q\Conf_{V, \Cbb}^r
    \]
    We can show that $ [\mobius_q\mobius_p^{-1}(A), \iota] \in \Conf_{V, \Cbb}^r $ by applying the following \cref{lem:relation} to the case that $ \mobius = \mobius_q\mobius_p^{-1} $. 
\end{enumerate}

\begin{lemma}\label{lem:relation}
    Let $ \mobius(z) = (az + b)(cz + d)^{-1} $ be a M\"obius transformation and $ c \neq 0 $. 
    Let $ [A, \iota] \in \Conf_{V, \Cbb}^r $, suppose that $ - c^{-1}d $ is not an eigenvalue of $ A $. 
    Then, $ [\mobius(A), \iota] \in \Conf_{V, \Cbb}^r $ holds. 
\end{lemma}

\begin{proof}
    Since $ cA + d $ is invertible, $ \mobius(A) = (aA + b)(cA + d)^{-1} $ is well-defined. 
    We will show that 
    \[
        \sum_{i = 0}^{r - 1} \mobius(A)^i \iota \colon V^{\oplus r} \to \Cbb^r
    \]
    is surjective. 
    By assumption, 
    \[
        \sum_{i = 0}^{r - 1} A^i \iota \colon V^{\oplus r} \to \Cbb^r
    \]
    is a surjection. 
    That is to say
    \[
        q_1 \colon \Cbb[z] \otimes V \to \Cbb^r ; \  \sum_{i = 0}^k z^i \otimes v_i \mapsto \sum_{i = 0}^k A^i \iota(v_i)
    \]
    is a surjection. 
    Define
    \[
        q_2 \colon \Cbb[z] \otimes V \to \Cbb^r ; \  \sum_{i = 0}^k z^i \otimes v_i \mapsto \sum_{i = 0}^k \mobius(A)^i \iota(v_i). 
    \]
    Using the characteristic polynomial of $ -c\mobius(A) + a $ and the fact that $ -c\mobius(A) + a = (ad - bc)(cA + d)^{-1} $ is invertible, we can see that there exists a polynomial $ g(z) \in \Cbb[z] $ such that $ (-c\mobius(A) + a)^{-1} = g(\mobius(A)) $. 
    Define
    \[
        \psi \colon \Cbb[z] \otimes V \to \Cbb[z] \otimes V ; \  \sum_{i = 0}^k z^i \otimes v_i \mapsto \sum_{i = 0}^k ((dz - b)g(z))^i \otimes v_i. 
    \]
    Then, $ q_1 = q_2 \circ \psi $. 
    In fact, 
    \begin{align*}
        q_2 \circ \psi\left(\sum_{i = 0}^k z^i \otimes v_i\right) & = q_2\left(\sum_{i = 0}^k ((dz - b)g(z))^i \otimes v_i\right) \\
        & = \sum_{i = 0}^k ((d\mobius(A) - b)g(\mobius(A)))^i \iota(v_i) \\
        & = \sum_{i = 0}^k ((d\mobius(A) - b)(-c\mobius(A)+ a)^{-1})^i \iota(v_i) \\
        & = \sum_{i = 0}^k A^i \iota(v_i) = q_1\left(\sum_{i = 0}^k z^i \otimes v_i\right). 
    \end{align*}
    Thus, $ q_1 $ is surjective, so is $ q_2 $. 
    Since degree of the characteristic polynomial $ \mobius(A) $, the surjectivity to be shown holds. 
\end{proof}

\begin{definition}
    Let $ V $-configuration space $ \Conf_{V, \Cbb P^1}^r $ of rank $ r $ over $ \Cbb P^1 $ be defined by 
    \[
        \Conf_{V, \Cbb P^1}^r \coloneqq \left({}^\infty\Conf_{V, \Cbb}^r \sqcup \coprod_{p \in \Cbb} {}^p\Conf_{V, \Cbb}^r\right) \Big/ {\sim}. 
    \]
    The topology is the quotient topology of the direct sum topology. 
    Let $ U $ be an open set of $ \Cbb P^1 $. 
    Let $ V $-configuration space $ \Conf_{V, U}^{r} \subset \Conf_{V, \Cbb P^1}^r $ of rank $ r $ over $ U $ be defined by 
    \[
        \Conf_{V, U}^r \coloneqq {}^{\infty}\Conf_{V, U \setminus \{\infty\}}^r \cup \bigcup_{p \in \Cbb} {}^p\Conf_{V, \mobius_p(U) \setminus \{\infty\}}^r \subset \Conf_{V, \Cbb P^1}^r. 
    \]
    Here, we regard $ {}^\infty\Conf_{V, U \setminus \{\infty\}}^r \subset {}^\infty\Conf_{V, \Cbb}^r $ and $ {}^p\Conf_{V, \mobius_p(U) \setminus \{\infty\}}^r \subset {}^p\Conf_{V, \Cbb}^r $ for $ p \in \Cbb $. 
\end{definition}
    
\subsection{Restriction of configuration spaces}\label{subsec:restriction}

The contents in this section will be used only in the proof of \cref{thm:coker-conti2}. 
The aim of this section is defining $ \Conf_{V, \Cbb P^1 \setminus \gamma}^{r, dl - r} $ and the restriction maps of configuration spaces 
\[
    \res_{D_\gamma} \colon \Conf_{V, \Cbb P^1 \setminus \gamma}^{r, dl - r} \to \Conf_{V, D_\gamma}^r.
\] 
Let $ d, l \geq 0 $ and $ 0 \leq r \leq dl $. 
Let $ \gamma $ be a simple closed curve in $ \Cbb $. 
Let $ D_{\gamma} $ be the bounded one of the two connected components of $ \Cbb \setminus \gamma $. 

\begin{definition}
    Let 
    \[
        \Conf_{V, \Cbb \setminus \gamma}^{r, dl - r} \subset \Conf_{V, \Cbb \setminus \gamma}^{dl}
    \]
    denote the subspace of $ \Conf_{V, \Cbb \setminus \gamma}^{dl} $ which consists of $ [A, \iota] \in \Conf_{V, \Cbb \setminus \gamma}^{dl} $ such that the number of roots with multiplicity in $ D_\gamma $ of the characteristic polynomial of $ A $ is $ r $. 
    The subspaces of $ {}^\infty\Conf_{V, \Cbb \setminus \gamma}^{dl} $ and $ {}^p\Conf_{V, \Cbb \setminus \gamma}^{dl} $ corresponding to $ \Conf_{V, \Cbb \setminus \gamma}^{r, dl - r} \subset \Conf_{V, \Cbb \setminus \gamma}^{dl} $ are denoted by $ {}^\infty\Conf_{V, \Cbb \setminus \gamma}^{r, dl - r} $, $ {}^p\Conf_{V, \Cbb \setminus \gamma}^{r, dl - r} $, respectively. 
    Note that 
    \[
        \Conf_{V, \Cbb P^1 \setminus \gamma}^{dl} = {}^\infty\Conf_{V, \Cbb \setminus \gamma}^{dl} \cup \bigcup_{p \in \Cbb \setminus \bar{D_\gamma}} {}^p\Conf_{V, \Cbb \setminus \mobius_p(\gamma)}^{dl}. 
    \]
    The subspace $ \Conf_{V, \Cbb P^1 \setminus \gamma}^{r, dl - r} \subset \Conf_{V, \Cbb P^1 \setminus \gamma}^{dl} $ is defined by 
    \[
        \Conf_{V, \Cbb P^1 \setminus \gamma}^{r, dl - r} \coloneqq {}^\infty\Conf_{V, \Cbb \setminus \gamma}^{r, dl - r} \cup \bigcup_{p \in \Cbb \setminus \bar{D_\gamma}} {}^p\Conf_{V, \Cbb \setminus \mobius_p(\gamma)}^{r, dl - r}. 
    \]
\end{definition}

We define a map $ \res_{D_\gamma} \colon \Conf_{V, \Cbb \setminus \gamma}^{r, dl - r} \to \Conf_{V, D_\gamma}^r $ as follows: 
Take $ [A, \iota] \in \Conf_{V, \Cbb \setminus \gamma}^{r, dl - r} $. 
Let $ W \subset \Cbb^{dl} $ denote the direct sum of generalized eigenspaces corresponding to eigenvalues of $ A $ included in $ D_\gamma $. 
Also, let $ W' \subset \Cbb^{dl} $ denote the direct sum of generalized eigenspaces corresponding to eigenvalues of $ A $ that are not included in $ D_\gamma $. 
We can see that $ \dim_\Cbb W = r $, $ \dim_\Cbb W' = dl - r $. 
Let $ A_W \colon W \to W $ be the restriction of $ A \colon \Cbb^{dl} \to \Cbb^{dl} $ on $ W $. 
Let $ \proj_W \colon \Cbb^{dl} \to W $ be the projection with respect to the direct sum decomposition $ \Cbb^{dl} = W \oplus W' $. 
Then, we define $ \res_{D_\gamma}([A, \iota]) \coloneqq [A_W, \proj_W \circ \iota \colon V \to W] \in \Conf_{V, D_\gamma}^r $. 

\begin{lemma}\label{lem:proj-conti}
    The map $ \res_{D_\gamma} \colon \Conf_{V, \Cbb \setminus \gamma}^{r, dl - r} \to \Conf_{V, D_\gamma}^r $ is continuous. 
\end{lemma}

\begin{proof}
    We defined a vector bundle $ \mathcal{E}_{V, \Cbb \setminus \gamma}^{dl} $ over $ \Conf_{V, \Cbb \setminus \gamma}^{dl} $ and bundle homomorphisms
    \[
        A^{dl} \coloneqq A_{V, \Cbb \setminus \gamma}^{dl} \colon \mathcal{E}_{V, \Cbb \setminus \gamma}^{dl} \to \mathcal{E}_{V, \Cbb \setminus \gamma}^{dl}, \quad
        \iota^{dl} \coloneqq \iota_{V, \Cbb \setminus \gamma}^{dl} \colon \underline{V} \to \mathcal{E}_{V, \Cbb \setminus \gamma}^{dl}
    \]
    in \cref{subsec:construction}. 
    Let $ \mathcal{E}_{V, \Cbb \setminus \gamma}^{r, dl - r} $ denote the restriction of the vector bundle $ \mathcal{E}_{V, \Cbb \setminus \gamma}^{dl} \to \Conf_{V, \Cbb \setminus \gamma}^{dl} $ on $ \Conf_{V, \Cbb \setminus \gamma}^{r, dl - r} $. 
    Also, let $ A^{r, dl - r} \colon \mathcal{E}_{V, \Cbb \setminus \gamma}^{r, dl - r} \to \mathcal{E}_{V, \Cbb \setminus \gamma}^{r, dl - r} $ and $ \iota^{r, dl - r} \colon \underline{V} \to \mathcal{E}_{V, \Cbb \setminus \gamma}^{r, dl - r} $ denote the restriction of $ A^{dl} $ and $ \iota^{dl} $, respectively. 
    For simplicity of notation, we write $ X \coloneqq \Conf_{V, \Cbb \setminus \gamma}^{r, dl - r} $ and $ E \coloneqq \mathcal{E}_{V, \Cbb \setminus \gamma}^{r, dl - r} $. 
    By construction, for each $ x \in X $, $ A^{r, dl - r}_x $ does not have eigenvalues on $ \gamma $ and the sum of dimensions of generalized eigenspaces corresponding to eigenvalues included inside $ \gamma $ is $ r $. 
    Let $ E_1 $ denote the subbundle of $ E $ which consists of generalized eigenspaces corresponding to eigenvalues inside $ \gamma $. 
    Also, let $ E_2 $ denote the subbundle  of $ E $ which consists of generalized eigenspaces corresponding to eigenvalues outside $ \gamma $. 
    Note that $ E_1 \oplus E_2 = E $. 
    By the direct sum decomposition $ \End(E) = \End(E_1) \oplus \Hom(E_1, E_2) \oplus \Hom(E_2, E_1) \oplus \End(E_2) $, $ \End(E_1) $ is regarded as a subbundle of $ \End(E) $. 
    $ A^{r, dl - r} $ is a continuous section of $ \End(E) $. 
    By definition, the matrix representation of $ A^{r, dl - r} $ with respect to the direct sum decomposition $ E = E_1 \oplus E_2 $ can be written as 
    \[
        A^{r, dl - r} = 
        \begin{pmatrix}
            A^{r, dl - r}_1 & 0 \\
            0 & A^{r, dl - r}_2
        \end{pmatrix}. 
    \]
    Here, $ A^{r, dl - r}_1 $ and $ A^{r, dl - r}_2 $ are (not necessarily continuous) sections of $ \End(E_1) $ and $ \End(E_2) $, respectively. 
    For any $ x \in X $, we can write
    \[
        \begin{pmatrix}
            A^{r, dl - r}_{1, x} & 0 \\
            0 & 0
        \end{pmatrix}
        = \frac{1}{2 \pi i} \int_\gamma z \cdot (z \cdot \id - A^{r, dl - r}_x)^{-1} dz, 
    \]
    they are continuous for $ x $. 
    Thus, $ A^{r, dl - r}_1 $ is a continuous section of $ \End(E_1) $. 
    Also, $ \iota^{r, dl - r} $ is a continuous section of $ \Hom(\underline{V}, E) $ and the vector representation of it with respect to $ E = E_1 \oplus E_2 $ can be written as 
    \[
        \iota^{r, dl - r} = 
        \begin{pmatrix}
            \iota^{r, dl - r}_1 \\
            \iota^{r, dl - r}_2
        \end{pmatrix}. 
    \]
    Here, $ \iota^{r, dl - r}_1 $ and $ \iota^{r, dl - r}_2 $ are (not necessarily continuous) section of $ \Hom(\underline{V}, E_1) $ and $ \Hom(\underline{V}, E_2) $, respectively. 
    The composition
    \[
        \begin{pmatrix}
            \iota^{r, dl - r}_1 \\
            0
        \end{pmatrix}
    \]
    of $ \iota^{r, dl - r} $ and the first projection is continuous for $ x \in X $, so $ \iota^{r, dl - r}_1 $ is a continuous section of $ \Hom(\underline{V}, E_1) $. 
    From the above, we get a vector bundle $ E_1 $ of rank $ r $ over $ X $ and bundle homomorphisms $ A^{r, dl - r}_1 \colon E_1 \to E_1 $ and $ \iota^{r, dl - r}_1 \colon \underline{V} \to E_1 $. 
    By \cref{lem:configuration-space-one-to-one}, these define a continuous map $ X \to \Conf_{V, D_\gamma}^r $ and it coincides with $ \res_{D_\gamma} \colon X \to \Conf_{V, D_\gamma}^r $ by definition. 
\end{proof}

\begin{remark}
    \cref{lem:proj-conti} can also be proved in the following way: 
    Define a map $ \res_{\Cbb \setminus \bar{D_\gamma}} \colon \Conf_{V, \Cbb \setminus \gamma}^{r, dl - r} \to \Conf_{V, \Cbb \setminus \bar{D_\gamma}}^{dl - r} $ in the same way as $ \res_{D_\gamma} \colon \Conf_{V, \Cbb \setminus \gamma}^{r, dl - r} \to \Conf_{V, D_\gamma}^r $.
    It is easy to see that the map
    \[
        (\res_{D_\gamma}, \res_{\Cbb \setminus \bar{D_\gamma}}) \colon \Conf_{V, \Cbb \setminus \gamma}^{r, dl - r} \to \Conf_{V, D_\gamma}^r \times \Conf_{V, \Cbb \setminus \bar{D_\gamma}}^{dl - r}
    \]
    has a continuous inverse map. 
    By \cref{prop:global-section-bundle}, this inverse map is a continuous bijection between same $ dl \cdot \dim_{\Cbb}{V} $-dimensional complex manifolds, so by invariance of domain, $ (\res_{D_\gamma}, \res_{\Cbb \setminus \bar{D_\gamma}}) $ is also continuous. 
    Therefore, $ \res_{D_\gamma} $ is continuous. 
\end{remark}

\begin{definition}
    A continuous map $ \res_\infty $ is defined by 
    \[
        \res_\infty \coloneqq \res_{D_\gamma} \colon {}^\infty\Conf_{V, \Cbb \setminus \gamma}^{r, dl - r} \to {}^\infty\Conf_{V, D_\gamma}^r. 
    \]
    Also, for each $ p \in \Cbb \setminus \bar{D_\gamma} $, $ \res_p $ is defined as a composition 
    \[
        \res_p \colon {}^p\Conf_{V, \Cbb \setminus \mobius_p(\gamma)}^{r, dl - r} \xrightarrow{\res_{D_{\mobius_p(\gamma)}}} {}^p\Conf_{V, D_{\mobius_p(\gamma)}}^r \xrightarrow{\mobius_p^{-1}} \Conf_{V, D_\gamma}^r. 
    \]
    The family of continuous maps $ \{\res_p\}_{p \in \Cbb P^1 \setminus \bar{D_\gamma}} $ is glued together and a continuous map $ \res_{D_\gamma} \colon \Conf_{V, \Cbb P^1 \setminus \gamma}^{r, dl - r} \to \Conf_{V, D_\gamma}^r $ is defined. 
\end{definition}

\subsection{Comparison with Quot schemes}

In this section, we will compare $ \Conf_{V, \Cbb}^r $ with Quot schemes in algebraic geometry. 
Quot shcemes are the representations of some presheaves over the category of schemes. 
We follow the definition of Quot schemes to \cite{MR2223407}. 

\begin{definition}
    Let $ Sch_\Cbb $ be the category of locally noetherian schemes over $ \Spec{\Cbb} $
    Let $ r \in \Zbb_{\geq 0} $ and $ E $ be a coherent sheaf on $ \Abb_\Cbb^1 $. 
    A contravariant functor $ \Qfrak uot_{E, \Abb_\Cbb^1}^r \colon Sch_\Cbb^{\op} \to Set $ is defined as follows:
    For $ T \to \Spec {\Cbb} $, a pair $ (\Fcal, q) $ is called a family of quotients of $ E $ parametrized by $ T $ if it satisfies that
    \begin{enumerate}
        \item $ \Fcal $ is a coherent sheaf over $ \Abb_T^1 $ such that its schematic support is proper over $ T $ and it is flat over $ T $, 
        \item $ q \colon E_T \to \Fcal $ is surjective $ \Ocal_{\Abb^1_T} $-linear homomorphism where $ E_T $ is pullback of $ E $ along the projection $ \Abb^1_T \to \Abb^1_\Cbb $ and
        \item for each $ t \in T $, the Hirbert polynomial of $ \Fcal_t = \rest{\Fcal}{(\Abb^1_T)_t} $ is constant $ r $. 
    \end{enumerate}
    We define $ \Qfrak uot_{E, \Abb_\Cbb^1}^r(T) $ as the set of isomorphism classes of families of quotients of $ E $ parametrized by $ T $. 
\end{definition}

\begin{fact}[\cite{MR0555258}]
    The functor $ \Qfrak uot_{E, \Abb_\Cbb^1}^r $ is representable. 
\end{fact}

We write $ \mathrm{Quot}_{E, \Abb_\Cbb^r}^r $ for the representation of the functor $ \Qfrak uot_{E, \Abb_\Cbb^1}^r $. 
In the case of $ E = \Ocal_{\Abb^1_\Cbb}^d $, by doing the same construction as in $ \Conf_{\Cbb^d, \Cbb}^r $ algebraically, we can describe the representation more concretely as follows. 
For the following parts of this section, see \cite{ricolfi2023motivic} for detail. 

Let $ (r^2 + rd) $-dimensional affine space $ \Abb $ be defined by 
\[
    \Abb \coloneqq M_r(\Cbb) \times \Hom_\Cbb(\Cbb^d, \Cbb^r). 
\]
A left action on $ \Abb $ of $ \GL_r(\Cbb) $ is defined by $ g \cdot (A, \iota) \coloneqq (gAg^{-1}, g\iota) $ for $ g \in \GL_r(\Cbb) $ and $ (A, \iota) \in \Abb $. 
Let $ U_d^r $ denote the open subscheme of $ \Abb $ which consists of points $ (A, \iota) \in \Abb $ such that $ \sum_{i = 0}^\infty A^i \iota \colon (\Cbb^d)^{\oplus \infty} \to \Cbb^r $ is surjective. 
The set of GIT stable points in $ \Abb $ with respect to the character $ \det \colon \GL_r(\Cbb) \to \Cbb^{\times} $ coincides with $ U_d^r $. 
Then, let $ \GL_r(\Cbb) \backslash U_d^r $ be the GIT quotient. 

\begin{proposition}
    Let $ V $ be a finite-dimensional complex vector space. 
    Then, the analytification of $ \mathrm{Quot}_{\Ocal_{\Abb^1_\Cbb} \otimes V, \Abb_\Cbb^1}^r $ is homeomorphic to $ \Conf_{V, \Cbb}^r $. 
\end{proposition}

\begin{proof}
    By \cite[Lemma 2.1]{ricolfi2023motivic}, the quotient map $ U_d^r \to \GL_r(\Cbb) \backslash U_d^r $ is principal $ \GL_r(\Cbb) $ bundle with respect to \'{e}tale topology. 
    Thus, its analytification is also principal $ \GL_r(\Cbb) $ bundle, so the analytification of $ \GL_r(\Cbb) \backslash U_d^r $ is homeomorphic to $ \Conf_{\Cbb^d, \Cbb}^r $. 
    On the other hand, the GIT quotient $ \GL_r(\Cbb) \backslash U_d^r $ is the representation of $ \Qfrak uot_{\Ocal_{\Abb^1_\Cbb}^d, \Abb_\Cbb^1}^r $ by \cite[Proposition 2.3]{ricolfi2023motivic}. 
    From the above, it follows that the analytification of $ \mathrm{Quot}_{\Ocal_{\Abb^1_\Cbb}^d, \Abb_\Cbb^1}^r $ is homeomorphic to $ \Conf_{\Cbb^d, \Cbb}^r $. 
\end{proof}

\section{Construction and continuity of configuration maps}\label{sec:continuity}

\subsection{Definition of configuration maps}\label{subsec:configuration-map}

Let $ V $ be a $ d $-dimensional complex vector space and $ l $ and $ r $ non-negetive integers. 
Also, let $ \gamma $ be a simple closed curve in $ \Cbb $ and $ D_{\gamma} $ the bounded one of the two connected components of $ \Cbb \setminus \gamma $. 

\begin{definition}\label{def:poly}
    Let $ \Poly_{V}^l $ denote the set of $ \End_{\Cbb}(V) $-coefficient polynomials of degree $ l $. 
    We can naturally define the topology of $ \Poly_{V}^l $ from $ \End_{\Cbb}(V)^{\times (l + 1)} $. 
    We write $ \Poly_{V}^{l, inv} $ for the subspace of $ \Poly_{V}^l $ consisting of polynomials whose coefficient of degree $ l $ is invertible. 
    Also, let us denote by $ \Poly_{V, \gamma}^{l, r} $ the subspace of $ \Poly_{V}^l $ consisting of $ f(z) \in \Poly_{V}^l $ which satisfy that
    \begin{enumerate}
        \item for any $ \lambda \in \gamma, f(\lambda) \in \GL_{\Cbb}(V) $ and
        \item the number of roots with multiplicity of $ \det{f(z)} $ included in $ D_{\gamma} $ is $ r $. 
    \end{enumerate}
\end{definition}

To define the configuration map $ \conf_{\Cbb} \colon \Poly_{V}^{l, inv} \to \Conf_{V, \Cbb}^{dl} $, we will prepare some lemmas. 
Take $ f(z) \in \Poly_{V}^{l, inv} $. 
Let $ \Ocal_{\Cbb} $ be the sheaf of holomorphic function on $ \Cbb $ and
$ \Mcal_{\Cbb} $ the sheaf of meromorphic function on $ \Cbb $. 
Let $ \Ocal_{\Cbb}(\textendash; V) $ be the sheaf of $ V $ valued holomorphic function on $ \Cbb $ and 
$ \Mcal_{\Cbb}(\textendash; V) $ the sheaf of $ V $ valued meromorphic function on $ \Cbb $.
Then, $ \Ocal_{\Cbb}(\textendash; V) \cong \Ocal_{\Cbb} \otimes V $ and $ \Mcal_{\Cbb}(\textendash; V) \cong \Mcal_{\Cbb} \otimes V $. 
We define $ \tilde{f} \colon \Ocal_{\Cbb}(\textendash; V) \to \Ocal_{\Cbb}(\textendash; V) $ by $ \tilde{f}(U)(s)(z) \coloneqq f(z)s(z) $ for an open set $ U $ of $ \Cbb $ and a section $ s \in \Ocal_{\Cbb}(U; V) $. 
We can also define the extention $ \tilde{f}^m \colon \Mcal_{\Cbb}(\textendash; V) \to \Mcal_{\Cbb}(\textendash; V) $ of $ \tilde{f} $. 

\begin{lemma}\label{lem:inj}
    The homomorphism $ \tilde{f}^m $ is $ \Mcal_{\Cbb} $-isomorphism. 
    In paticular, $ \tilde{f} $ is injective. 
\end{lemma}

\begin{proof}
    Let $ \lambda \in \Cbb $. 
    The stalk of $ \tilde{f}^m $ is regarded as $ \tilde{f}^m_\lambda \colon \Mcal_{\Cbb, \lambda}^d \to \Mcal_{\Cbb, \lambda}^d $. 
    Since $ f(z) \in \Poly_{V}^{l, inv} $, 
    we can see that the coefficient of maximum degree of $ \det{f(z)} $ is not 0, in paticular, 
    $ \det{\tilde{f}^m_\lambda} \neq 0 $. 
    Then, it follows that $ \tilde{f}^m_\lambda $ is $ \Mcal_{\Cbb, \lambda} $-isomorphism from the fact that $ \Mcal_{\Cbb, \lambda} $ is a field. 
    Therefore, $ \tilde{f}^m $ is $ \Mcal_{\Cbb} $-isomorphism. 
\end{proof}

By \cref{lem:inj}, we get the following short exact sequence:
\[
    \begin{tikzcd}
        0 \ar[r] & \Ocal_{\Cbb}(\textendash; V) \ar[r, "\tilde{f}"] & \Ocal_{\Cbb}(\textendash; V) \ar[r] & \Coker{\tilde{f}} \ar[r] & 0
    \end{tikzcd}
\]

\begin{lemma}\label{lem:support}
    The support of $ \Coker{\tilde{f}} $ is a finite set and $ \dim_{\Cbb}{\Coker{\tilde{f}}_\lambda} < \infty $ for any $ \lambda \in \Cbb $. 
    Furthermore, $ \dim_{\Cbb}{H^0(\Coker{\tilde{f}})} = dl $. 
\end{lemma}

\begin{proof}
    Since $ \det{f(z)} $ has a finite number of roots, 
    there are finitely many $ \lambda $'s in $ \Cbb $ for which $ \tilde{f}_\lambda $ is not invertible. 
    Thus, there are finitely many $ \lambda $'s in $ \Cbb $ for which stalk $ \Coker{\tilde{f}}_\lambda $ is not 0. 
    Let $ \Zcal $ be the zero set of $ \det{f(z)} $. 
    Fix an isomorphism $ V \cong \Cbb^d $. 
    For $ \lambda \in \Zcal $, we can consider $ \tilde{f}_\lambda \in M_d(\Ocal_{\Cbb, \lambda}) $. 
    Since $ \Ocal_{\Cbb, \lambda} $ is a PID and $ \tilde{f}_\lambda $ is injective, by invariant factor theory, there are $ P_\lambda, Q_\lambda \in \GL_d{(\Ocal_{\Cbb, \lambda})} $ and non-negative integers $ e_{\lambda, 1} \leq \dots \leq e_{\lambda, d} $ which satisfy
    \[
        Q_\lambda \tilde{f}_\lambda P_\lambda =
        \begin{pmatrix}
            (z - \lambda)^{e_{\lambda, 1}} & & \\
            & \ddots & \\
            & & (z - \lambda)^{e_{\lambda, d}}
        \end{pmatrix}. 
    \]
    Because $ \det{\tilde{f}} $ and $ \det{\tilde{f}_\lambda} $ are equal up to unit multiplication of $ \Ocal_{\Cbb, \lambda} $, 
    the dimension $ \dim_{\Cbb}{\Coker{\tilde{f}_\lambda}} = e_{\lambda, 1} + \dots + e_{\lambda, d} $ is equal to the multiplicity of the root $ \lambda $ of $ \det{f(z)} $. 
    In paticular, it is finite. 
    Therefore, it follows that 
    \[
        \dim_{\Cbb}{H^0(\Coker{\tilde{f}})} = \sum_{\lambda \in \Zcal} \dim_{\Cbb}{\Coker{\tilde{f}_\lambda}} = \operatorname{deg}{(\det{f(z)})} = dl.
    \] 
\end{proof}

\begin{lemma}\label{lem:surj}
    The morphism $ \pi \colon \Ocal_{\Cbb}(\Cbb; V) \to H^0(\Coker{\tilde{f}}) $ that the canonical projection $ \Ocal_{\Cbb}(\textendash; V) \to \Coker{\tilde{f}} $ induces into $ H^0 $ is surjective. 
\end{lemma}

\begin{proof}
    We use the notations used in the proof of \cref{lem:support}. 
    The strategy is to take pullback of $ \Ocal_{\Cbb}(\Cbb; V) $ along $ \tilde{f}^m(\Cbb) $ and consider it in $ \Mcal_{\Cbb}(\Cbb; V) $. 
    For $ \lambda \in \Zcal $, let $ \Mcal_{\Cbb, \lambda} $-homomorphism $ \psi_\lambda \colon \Mcal_{\Cbb, \lambda} \otimes V \to \Mcal_{\Cbb, \lambda}^d $ be defined as the composition
    \[
        \psi_\lambda \colon \Mcal_{\Cbb, \lambda} \otimes V \underset{\cong}{\to} \Mcal_{\Cbb, \lambda}^d \underset{\cong}{\overset{P_\lambda^{-1}}{\to}} \Mcal_{\Cbb, \lambda}^d. 
    \]
    The first isomorphism is induced by $ V \cong \Cbb^d $. 
    Let $ \Mcal_{\Cbb}(\Cbb; V)^{\tilde{f}} \subset \Mcal_{\Cbb}(\Cbb; V) $ denote the subspace of $ s $'s in $ \Mcal_{\Cbb}(\Cbb; V) $ which satisfy the following conditions:
    \begin{itemize}
        \item $ s $ is holomorphic on $ \Cbb \setminus \Zcal $. 
        \item For each $ \lambda \in \Zcal $ and $ 1 \leq j \leq d $, $ \psi_{\lambda}(s_\lambda)_j $ has a pole of order at most $ e_{\lambda, j} $. 
    \end{itemize}
    Then, it is easy to see that $ \Mcal_{\Cbb}(\Cbb; V)^{\tilde{f}} = \tilde{f}^m(\Cbb)^{-1}(\Ocal_{\Cbb}(\Cbb; V)) $. 
    Therefore, the surjectivity to be shown is reduced to a Mittag-Leffler type extention problem. 
    For each $ \lambda \in \Zcal $, take $ t_\lambda $ as a meromorphic function defined in a neighborhood of $ \lambda $ and suppose that it is holomorphic except for $ \lambda $ and $ \psi_{\lambda}(t_\lambda)_j $ has a pole of order at most $ e_{\lambda, j} $ for $ 1 \leq j \leq d $. 
    We can replace $ t_\lambda $ with one which has same germ and satisfies that each $ \psi_{\lambda}(t_\lambda)_j $ is a polynomial of $ (z - \lambda)^{-1} $. 
    Since $ \Zcal $ is a finite set, the sum $ s \coloneqq \sum_{\lambda \in \Zcal} t_\lambda $ is in $ \Mcal_{\Cbb}(\Cbb; V)^{\tilde{f}} $ and $ s $ and $ t_\lambda $ have same germ at $ \lambda $ for each $ \lambda $. 
\end{proof}

By \cref{lem:surj}, the following sequence is exact and by \cref{lem:support}, we have $ H^0(\Coker{\tilde{f}}) \cong \Cbb^{dl} $. 
Fix an isomorphism$ \varphi \colon H^0(\Coker{\tilde{f}}) \to \Cbb^{dl} $. 
\[
    \begin{tikzcd}
        0 \ar[r] & \Ocal_{\Cbb}(\Cbb; V) \ar[r, "\tilde{f}(\Cbb)"] & \Ocal_{\Cbb}(\Cbb; V) \ar[r, "\pi"] & H^0(\Coker{\tilde{f}}) \ar[r] & 0
    \end{tikzcd}
\]
The homomorphism $ \tilde{f} $ commutes with the multiplication operator $ z $ on $ \Ocal_{\Cbb}(\Cbb; V) $, so $ z \colon \Ocal_{\Cbb}(\Cbb; V) \to \Ocal_{\Cbb}(\Cbb; V) $ induces a homomorphism
\[
    \bar{z} \colon H^0(\Coker{\tilde{f}}) \to H^0(\Coker{\tilde{f}}). 
\]
Define $ A(f) \coloneqq \varphi \circ \bar{z} \circ \varphi^{-1} \in M_r(\Cbb) $. 
Also, define $ \iota(f) \coloneqq \varphi \circ \pi \circ \mathrm{const.} \colon V \to \Cbb^{dl} $ where $ \mathrm{const.} \colon V \to \Ocal_{\Cbb}(\Cbb; V) $ is the map sends to constant functions. 

\begin{definition}
    Let the configuration map $ \conf_{\Cbb} \colon \Poly_{V}^{l, inv} \to \Conf_{V, \Cbb}^{dl} $ be defined by $ \conf_{\Cbb}(f(z)) \coloneqq [A(f), \iota(f)] $. 
\end{definition}

We prepare the following lemmas to show well-definedness of this map. 

\begin{lemma}\label{lem:Chinese}
    Let $ g(z) $ be a non-zero $ \Cbb $-coefficient polynomial. 
    Then, the homomorphism $ \Cbb[z] \to \Ocal_{\Cbb}(\Cbb) / (g(z)) $ defined by the inclusion is surjective. 
\end{lemma}

\begin{proof}
    Factorize as $ g(z) = a_0 \prod_{i = 1}^k(z - \lambda_i)^{m_i} $. 
    Suppose $ \{\lambda_i\}_i $ are mutually distinct. 
    By Riemann's removable singularity theorem, $ (z - \lambda_i) $ is a maximal ideal of $ \Ocal_\Cbb(\Cbb) $. 
    Thus, by Chinese remainder theorem, we have 
    \[
        \Ocal_{\Cbb}(\Cbb) / (g(z)) \cong \Ocal_{\Cbb}(\Cbb) / \prod_{i = 1}^k((z - \lambda_i))^{m_i} \cong \prod_{i = 1}^k \Ocal_{\Cbb}(\Cbb) / ((z - \lambda_i))^{m_i}. 
    \]
    Also, we have 
    \begin{align*}
        \Cbb[z] / (g(z)) \cong \Cbb[z] / \prod_{i = 1}^k((z - \lambda_i))^{m_i} \cong \prod_{i = 1}^k \Cbb[z] / ((z - \lambda_i))^{m_i}. 
    \end{align*}
    Since for each $ i $, it follows that 
    \[
        \Ocal_{\Cbb}(\Cbb) / ((z - \lambda_i))^{m_i} \cong \bigoplus_{j = 0}^{m_i - 1} \Cbb \cdot (z - \lambda_i)^j \cong \Cbb[z] / ((z - \lambda_i))^{m_i}, 
    \]
    the homomorphism $ \Cbb[z] / (g(z)) \to \Ocal_\Cbb(\Cbb) / (g(z)) $ induced by inclusion is isomorphism. 
    Therefore, $ \Cbb[z] \to \Cbb[z] / (g(z)) \to \Ocal_\Cbb(\Cbb) / (g(z)) $ is surjective. 
\end{proof}

\begin{lemma}
    For $ f(z) \in \Poly_{V}^{l, inv} $, $ [A(f), \iota(f)] \in \Conf_{V, \Cbb}^{dl} $. 
\end{lemma}

\begin{proof}
    We need to show the surjectivity of $ \sum_{i = 0}^{dl - 1} A(f)^i \iota(f) \colon V^{\oplus dl} \to \Cbb^{dl} $. 
    To show this, it is sufficient to show that the composition
    \begin{equation}\label{eq:surj}
        \left\{\sum_{i = 0}^{dl - 1} a_i z^i \setmid a_i \in V\right\} \to \Ocal_{\Cbb}(\Cbb; V) \to H^0(\Coker{\tilde{f}}) 
    \end{equation}
    is surjective. 
    Let $ g(z) $ be the characteristic polynomial of $ \bar{z} $. 
    Since $ \deg{(g(z))} = dl $, 
    the surjectivity of (\ref{eq:surj}) follows if the composition
    \begin{equation}
        \Cbb[z] \otimes V \to \Ocal_{\Cbb}(\Cbb; V) \to H^0(\Coker{\tilde{f}})
    \end{equation}
    is surjective. 
    This follows from \cref{lem:Chinese} and the fact that $ H^0(\Coker{\tilde{f}}) $ is an $ \Ocal_{\Cbb}(\Cbb) / (g(z)) $-module. 
\end{proof}

It is easy to see that $ [A(f), \iota(f)] \in \Conf_{V, \Cbb}^{dl} $ does not depend on the choice of isomorphism $ \varphi \colon H^0(\Coker{\tilde{f}}) \to \Cbb^{dl} $. 

Also, we define the configuration map $ \conf_{D_\gamma} \colon \Poly_{V, \gamma}^{l, r} \to \Conf_{V, D_{\gamma}}^r $ in the same way as $ \conf_{\Cbb} $. 
Here, $ \Ocal_{D_\gamma}(\textendash; V) \coloneqq \rest{\Ocal_{\Cbb}(\textendash; V)}{D_\gamma} $ is used instead of $ \Ocal_{\Cbb}(\textendash; V) $. 

\subsection{Continuity of configuration maps}

The aim in this section is to prove the following two proposition. 

\begin{theorem}\label{thm:coker-conti1}
    The configuration map $ \conf_{\Cbb} \colon \Poly_{V}^{l, inv} \to \Conf_{V, \Cbb}^{dl} $ is continuous.  
\end{theorem}

\begin{theorem}\label{thm:coker-conti2}
    Let $ \gamma $ is a simple closed curve in $ \Cbb $. 
    Then, the configuration map $ \conf_{D_\gamma} \colon \Poly_{V, \gamma}^{l, r} \to \Conf_{V, D_\gamma}^r $ is continuous. 
\end{theorem}

These theorem will be proved in a way of algebraic geometry. 
First, let us prepare the scheme corresponding to $ \Poly_V^{l, inv} $. 

\begin{definition}
    Let a commutative ring $ P_d^l $ be defined by 
    \[
        P_d^l \coloneqq \Cbb[\{X_{ij}^{(k)}\}_{1 \leq i, j \leq d, 0 \leq k \leq l}, T]/(T \cdot \det(X_{ij}^{(l)}) - 1). 
    \]
\end{definition}

The analytification of $ \Spec{P_d^l} $ is homeomorphic to $ \Poly_{\Cbb^d}^{l, inv} $. 
To prove \cref{thm:coker-conti1}, we will construct the element of $ \Qfrak uot_{\Ocal_{\Abb_\Cbb^1}^d, \Abb_\Cbb^1}^r(\Spec{P_d^l}) $ which corresponds to $ \conf_{\Cbb} $. 
Then, it will give a morphism of schemes $ \Spec{P_d^l} \to \mathrm{Quot}_{\Ocal_{\Abb_\Cbb^1}^d, \Abb_\Cbb^1}^r $ and we will get a continuous map $ \conf_{\Cbb} \colon \Poly_{\Cbb^d}^{l, inv} \to \Conf_{\Cbb^d, \Cbb}^r $ by taking analytification. 

Let $ M_d(P_d^l) $-coefficient polynomial $ F(z) $ of order $ l $ whose coefficient of order $ l $ is invertible be defined by
\[
    F(z) \coloneqq (X_{ij}^{(l)})z^l + (X_{ij}^{(l - 1)})z^{l - 1} + \dots + (X_{ij}^{(0)}). 
\]
The polynomial $ F(z) $ defines a sheaf homomorphism $ \tilde{F} \colon \Ocal_{\Abb_{P_d^l}^1}^d \to \Ocal_{\Abb_{P_d^l}^1}^d $. 
Since $ \tilde{F} $ is injective, taking cokernel, we get a short exact sequence
\[
    \begin{tikzcd}
        0 \ar[r] & \Ocal_{\Abb_{P_d^l}^1}^d \ar[r, "\tilde{F}"] & \Ocal_{\Abb_{P_d^l}^1}^d \ar[r, "q"] & \Coker{\tilde{F}} \ar[r] & 0
    \end{tikzcd}. 
\]

\begin{proposition}\label{prop:Quot}
    We have $ [\Coker{\tilde{F}}, q] \in \Qfrak uot_{\Ocal_{\Abb_\Cbb^1}^d, \Abb_\Cbb^1}^{dl}(\Spec{P_d^l}) $.
\end{proposition}

\begin{proof}
    First, we show that the schematic support of $ \Coker{\tilde{F}} $ is proper over $ \Spec{P_d^l} $. 
    The support of $ \Coker{\tilde{F}} $ is contained in $ \Spec(P_d^l[z]/(\det{F(z)})) $. 
    Since $ P_d^l[z] /(\det(F(z))) $ is a finitely generated $ P_d^l $-module, the support of $ \Coker{\tilde{F}} $ is proper over $ \Spec{P_d^l} $. 

    Next, we show that $ \Coker{\tilde{F}} $ is flat over $ \Spec{P_d^l} $. 
    Let $ p \colon \Abb_{P_d^l} \to \Spec{P_d^l} $. 
    The coherent sheaf $ p_*(\Coker{\tilde{F}}) $ is locally free because $ \Spec{P_d^l} $ is noetherian and reduced and the rank of each stalk is a constant $ dl $, independent of points in $ \Spec{P_d^l} $\cite[Exercise II.5.8]{hartshorne1977algebraic}. 
    In paticular, $ \Coker{\tilde{F}} $ is flat over $ \Spec{P_d^l} $. 
    
    Finally, we show that the Hilbert polynomial of $ (\Coker{\tilde{F}})_f $ is constant $ dl $ for $ f \in \Spec{P_d^l} $. 
    This follows from the fact that the degree of $ F(z) $ is $ dl $. 
\end{proof}

\begin{proof}[Proof of \cref{thm:coker-conti1}]
    By \cref{prop:Quot}, we can define a morphism of schemes
    \[
        \Spec{P_d^l} \to \mathrm{Quot}_{\Ocal_{\Abb^1_\Cbb}^d, \Abb_\Cbb^1}^{dl}.
    \]
    By taking analytification, we get a continuous map
    \[
        \Poly_{\Cbb^d}^{l, inv} \to \Conf_{\Cbb^d, \Cbb}^{dl}
    \]
    and by construction, this coincides with the map $ \conf_{\Cbb} \colon \Poly_{\Cbb^d}^{l, inv} \to \Conf_{\Cbb^d, \Cbb}^{dl} $. 
\end{proof}

Next, we will define a continuous map $ \conf_{\Cbb P^1} \colon (\Poly_{V}^l) \setminus \{0\} \to \Conf_{V, \Cbb P^1}^{dl} $. 
For $ p \in \Cbb $, Let a subspace $ \Poly_{V, \Cbb P^1 \setminus \{p\}}^l \subset \Poly_{V}^l $ be defined by
\[
    \Poly_{V, \Cbb P^1 \setminus \{p\}}^l \coloneqq \{f(z) \in \Poly_{V}^l \mid f(p) \text{ is invertible. }\}. 
\]
Also, Let us define 
\[
    \Poly_{V, \Cbb P^1 \setminus \{\infty\}}^l \coloneqq \Poly_{V}^{l, inv}. 
\]
Then, 
\[
    (\Poly_{V}^l) \setminus \{0\} = \Poly_{V, \Cbb P^1 \setminus \{\infty\}}^l \cup \bigcup_{p \in \Cbb} \Poly_{V, \Cbb P^1 \setminus \{p\}}^l. 
\]
For $ p \in \Cbb P^1 $, we define $ \conf_{\Cbb P^1 \setminus \{p\}} \colon \Poly_{V, \Cbb P^1 \setminus \{p\}}^l \to {}^p\Conf_{V, \Cbb}^{dl} $ as follows:

\begin{itemize}
    \item If $ p = \infty $, define
    \[
        \conf_{\Cbb P^1 \setminus \{\infty\}} \coloneqq \conf_{\Cbb} \colon \Poly_{V, \Cbb P^1 \setminus \{\infty\}}^l = \Poly_{V}^{l, inv} \to {}^\infty\Conf_{V, \Cbb}^{dl}. 
    \]
    \item If $ p \in \Cbb $, for $ f(z) \in \Poly_{V, \Cbb P^1 \setminus \{p\}}^l $ we define
    \[
        \conf_{\Cbb P^1 \setminus \{p\}}(f(z)) \coloneqq \conf_{\Cbb}\left((-z)^lf\left(\mobius_p^{-1}(z)\right)\right) \in \Conf_{V, \Cbb}^{dl} = {}^p\Conf_{V, \Cbb}^{dl}. 
    \]
\end{itemize}

We can check well-definedness of $ \conf_{\Cbb P^1 \setminus \{p\}} $ for $ p \in \Cbb $ as follows:
For $ f(z) = \sum_{j = 0}^l a_j z^j $, 
\[
    (-z)^lf\left(\mobius_p^{-1}(z)\right) = (-z)^lf\left(\frac{1}{z} + p\right) = (-1)^l\sum_{j = 0}^l a_j z^{l - j}(pz + 1)^j, 
\]
so $ (-z)^lf\left(\mobius_p^{-1}(z)\right) \in \Poly_{V}^l $ and the coefficient of degree $ l $ is $ (-1)^lf(p) $, in paticular, invertible. 
By the following \cref{lem:gluing}, $ \{\conf_{\Cbb P^1 \setminus \{p\}}\}_{p \in \Cbb P^1} $ are glued together and define a continuous map $ \conf_{\Cbb P^1} \colon (\Poly_{V}^l) \setminus \{0\} \to \Conf_{V, \Cbb P^1}^{dl} $. 
To glue $ \conf_\Cbb $ and $ \conf_{\Cbb P^1 \setminus \{p\}} $ ($ p \in \Cbb $), we apply \cref{lem:gluing} for $ \mobius = \mobius_p $ and to glue $ \conf_{\Cbb P^1 \setminus \{p\}} $ and $ \conf_{\Cbb P^1 \setminus \{q\}} $ ($ p, q \in \Cbb $, $ p \neq q $), we apply it for $ \mobius = \mobius_q\mobius_p^{-1} $. 

\begin{lemma}\label{lem:gluing}
    Let $ \mobius(z) = (az + b)(cz + d)^{-1} $ be a M\"obius transformation and $ c \neq 0 $. 
    Then, the following diagram is commutative:
    \[
        \begin{tikzcd}[column sep=large]
            \left\{f(z) \in \Poly_V^{l, inv} \setmid f\left(-\frac{d}{c}\right) \text{ is invertible. }\right\} \arrow{r}{\conf_{\Cbb}} \arrow{d}{\Theta}
            & \Conf_{V, \Cbb \setminus \{-\frac{d}{c}\}}^{dl} \arrow{d} {\Xi}\\
            \left\{f(z) \in \Poly_V^{l, inv} \setmid f\left(\frac{a}{c}\right) \text{ is invertible. }\right\} \arrow{r}{\conf_{\Cbb}}
            & \Conf_{V, \Cbb \setminus \{\frac{a}{c}\}}^{dl}
        \end{tikzcd}
    \]
    Here, we define $ \Theta(f(z)) \coloneqq (-cz + a)^l f(\mobius^{-1}(z)) $ and $ \Xi([A, \iota]) \coloneqq ([\mobius(A), \iota]) $. 
\end{lemma}

\begin{proof}
    Take $ f(z) = \sum_{j = 0}^l a_j z^j \in \Poly_{V}^{l, inv} $ and suppose that $ f(-c^{-1}d) $ is invertible. 
    Let $ \conf_{\Cbb}(f(z)) = [A, \iota] $. 
    Then, we will show that $ \conf_{\Cbb}(\Theta(f(z))) = \Xi([A, \iota]) $. 
    By assumption, the following sequence is exact:
    \[
        \begin{tikzcd}
            0 \arrow{r} 
            & \Cbb[z] \otimes V \arrow{r}{f(z)} 
            & \Cbb[z] \otimes V \arrow{r}{q_1} 
            & \Cbb^{dl} \arrow{r} 
            & 0
        \end{tikzcd}
    \]
    Here, we define $ q_1(\sum_{i = 0}^k z^i \otimes v_i) \coloneqq \sum_{i = 0}^k A^i\iota(v_i) $. 
    It is enough to show that the following sequence is exact:
    \[
        \begin{tikzcd}
            0 \arrow{r} 
            & \Cbb[z] \otimes V \arrow{rrr}{(-cz + a)^l f(\mobius^{-1}(z))} 
            & 
            & 
            & \Cbb[z] \otimes V \arrow{r}{q_2} 
            & \Cbb^{dl} \arrow{r} 
            & 0
        \end{tikzcd}
    \]
    Here, we define $ q_2(\sum_{i = 0}^k z^i \otimes v_i) \coloneqq \sum_{i = 0}^k \mobius(A)^i\iota(v_i) $. 
    The surjectivity of $ q_2 $ follows from \cref{lem:relation}. 
    For $ v \in V $, 
    \begin{align*}
        q_2((-cz + a)^l f(\mobius^{-1}(z))v) & = q_2\left((-cz + a)^l \sum_{j = 0}^l a_j v \mobius^{-1}(z)^j\right) \\
        & = (-c\mobius(A) + a)^l \sum_{j = 0}^l A^j \iota(a_j v) \\
        & = (-c\mobius(A) + a)^l q_1(f(z)v) \\
        & = 0. 
    \end{align*}
    Thus, $ q_2 \colon (\Cbb[z] \otimes V) / {\Im((-cz + a)^l f(\mobius^{-1}(z)))} \to \Cbb^{dl} $ is well-defined and surjective, and since the dimension of domain and codomain are equal, it is isomorphism. 
    It means the exactness to be shown. 
\end{proof}

\begin{proof}[Proof of \cref{thm:coker-conti2}]
    By restriction of $ \conf_{\Cbb P^1} \colon (\Poly_{V}^l) \setminus \{0\} \to \Conf_{V, \Cbb P^1}^{dl} $, a continuous map 
    \[
        \conf_{\Cbb P^1 \setminus \gamma} \colon \Poly_{V, \gamma}^{l, r} \to \Conf_{V, \Cbb P^1 \setminus \gamma}^{r, dl - r}
    \]
    is defined. 
    By composition with a continuous map $ \res_{D_\gamma} \colon \Conf_{V, \Cbb P^1 \setminus \gamma}^{r, dl - r} \to \Conf_{V, D_\gamma}^r $ defined in \cref{subsec:restriction}, 
    a continuous map 
    \[
        \res_{D_\gamma} \circ \conf_{\Cbb P^1 \setminus \gamma} \colon \Poly_{V, \gamma}^{l, r} \to \Conf_{V, D_\gamma}^r
    \]
    is defined. 
    By construction, this corresponds to the map of proposition
    \[
        \conf_{D_\gamma} \colon \Poly_{V, \gamma}^{l, r} \to \Conf_{V, D_\gamma}^r. 
    \]
\end{proof}

\section{Proof of the Bott periodicity theorem} \label{sec:proof-of-bott-periodicity}

\begin{construction}\label{const:gluing}
    Let $ X $ be a compact Hausdorff space, $ E $ a complex vector bundle over $ X $. 
    Let $ l \in \Zbb_{\geq 0} $ and continuous sections $ a_j \in \Gamma(X, \End(E)) $ for $ 0 \leq j \leq l $ are given. 
    Then, we get a family of polynomials
    \[
        H = \left\{H_x(z) = \sum_{j = 0}^l a_j(x) z^j\right\}_{x \in X}. 
    \]
    For each $ x \in X $, $ z \in S^1 $, we suppose that $ H_x(z) \in \End_{\Cbb}(E_x) $ is invertible. 
    After this, we will construct a complex vector bundle $ \bundle_X(E, H) $ over $ X $ as follows:
    Let $ \{U_\lambda\}_\lambda $ be a trivialization covering of $ E $ and $ \rest{E}{U_\lambda} \cong U_\lambda \times V $ a trivialization for each $ \lambda $. 
    Here, $ V $ is a finite-dimensional complex vector space. 
    We define 
    \[
        H_\lambda \coloneqq \left\{(H_\lambda)_x(z) \coloneqq H_x(z) = \sum_{j = 0}^l a_j(x) z^j\right\}_{x \in U_\lambda}. 
    \]
    We can assume that the number of roots with multiplicity in $ \disk $ of $ \det(H_\lambda)_x $ is constant $ r_\lambda $, independent of $ x \in U_\lambda $.
    Applying \cref{const:basic} to $ H_\lambda $, we get a vector bundle $ \bundle_X(V, H_\lambda) $ over $ U_\lambda $. 
    Let $ \{g_{\mu\lambda} \colon U_\lambda \cap U_\mu \to \GL_{\Cbb}(V)\}_{\lambda, \mu} $ be transition maps of $ E $. 
    For each $ g_{\mu\lambda} $, if $ U_\lambda \cap U_\mu \neq \emptyset $, commutative diagram
    \[
        \begin{tikzcd}[column sep=huge]
            & (U_\lambda \cap U_\mu) \times \Ecal_{V, \disk}^{r_\lambda} \ar[d] \ar[r, "\bar{g'}_{\mu\lambda}"] & \Ecal_{V, \disk}^{r_\lambda} \ar[d] \\
            U_\lambda \cap U_\mu \ar[r, "\conf(H_\lambda)"] \ar[rr, bend right, "\conf(H_\mu)"] & (U_\lambda \cap U_\mu) \times \Conf_{V, \disk}^{r_\lambda} \ar[r, "g'_{\mu\lambda}"] & \Conf_{V, \disk}^{r_\lambda}
        \end{tikzcd}
    \]
    is induced. 
    Taking two pullback in the same way of \cref{const:basic}, we get an isomorphism $ \rest{\bundle_X(V, H_\lambda)}{U_\lambda \cap U_\mu} \cong \rest{\bundle_X(V, H_\mu)}{U_\lambda \cap U_\mu} $. 
    Gluing $ \{\bundle_X(V, H_\lambda)\}_\lambda $ by these isomorphisms for $ g_{\mu\lambda} $, we get a vector bundle $ \bundle_X(E, H) $ over $ X $. 
\end{construction}

\begin{lemma}\label{lem:property-of-bundle}
    \begin{enumerate}
        \item For $ f \colon Y \to X $, $ \bundle_Y(f^* E, f^* H) \cong f^* \bundle_X(E, H) $. 
        \item $ \bundle_X(E_1 \oplus E_2, H_1 \oplus H_2) \cong \bundle_X(E_1, H_1) \oplus \bundle_X(E_2, H_2) $. 
        \item $ \bundle_X(E, \id) = 0 $. 
        \item $ \bundle_X(E, zH) \cong E \oplus \bundle_X(E, H) $. 
        \item For a complex vector bundle $ G $ over $ X $, 
        \[
            \bundle_X(G \otimes E, \id \otimes H) \cong G \otimes \bundle_X(E, H). 
        \]
    \end{enumerate}
\end{lemma}

\begin{proof}
    It enough to show (1) in the case of trivial bundle $ E = \underline{V} $. 
    This follows from the following pullback diagram:
    \[
        \begin{tikzcd}[column sep=huge]
            \bundle_Y(\underline{V}, f^*H) \arrow{r} \arrow{d} \arrow[dr, phantom, pos=0.0495, "\lrcorner"] & \bundle_X(\underline{V}, H) \arrow{r} \arrow{d} \arrow[dr, phantom, pos=0.125, "\lrcorner"] & \Ecal_{V, \disk}^r \arrow{d} \\
            Y \arrow{r}{f} \ar[rr, bend right, "\conf(Hf)"] & X \arrow{r}{\conf(H)} & \Conf_{V, \disk}^r
        \end{tikzcd}
    \]
    It easy to show (2), (3) and (4) by construction. 
    It enough to show (5) in the case of trivial bundle $ G = \underline{\Cbb^n} $. 
    It reduces to
    \[
        \bundle_X(E^{\oplus n}, H^{\oplus n}) \cong \bundle_X(E, H)^{\oplus n}, 
    \]
    and this follows from (2). 
\end{proof}

\begin{proof}[Proof of \cref{thm:bott-periodicity}]
    From now on, we will define a homomorphism $ \alpha_X \colon K(X \times (D^2, S^1)) \to K(X) $. 
    Take $ [\bar{E}, \bar{F}, \bar{H}] \in K(X \times (D^2, S^1)) $. 
    Let $ E \coloneqq \rest{\bar{E}}{X \times \{1\}} $, then $ \bar{E} \cong \proj_X^*E $. 
    Also $ \bar{F} \cong \proj_X^*E $ since $ \rest{\bar{H}}{X \times \{1\}} \colon E = \rest{\bar{E}}{X \times \{1\}} \cong \rest{\bar{F}}{X \times \{1\}} $. 
    Thus, we can write $ [\bar{E}, \bar{F}, \bar{H}] = [\proj_X^*E, \proj_X^*E, H] $ by using a bundle morphism $ H \colon \rest{\proj_X^*E}{X \times S^1} \cong \rest{\proj_X^*E}{X \times S^1} $. 
    Furthermore, since $ X $ is compact, replacing by Cezaro means $ H_n $ of Fourier series of $ H $ (\cite{MR0178470}), we may assume that $ H $ has a form of a continuous family of Laurent polynomial
    \[
        H = \left\{H_x(z) = \sum_{j = - k}^l a_j(x) z^j\right\}_{x \in X}
    \]
    where $ k, l \in \Zbb_{\geq 0} $ and $ a_j(x) \in \End_{\Cbb}(E_x) $ for any $ x \in X $. 
    Multiplying $ H $ by $ z^k $, we have a family of polynomials
    \[
        z^k H = \left\{(z^kH)_x(z) = \sum_{j = 0}^{k + l} a_{j - k}(x) z^j\right\}_{x \in X}. 
    \]
    Applying \cref{const:gluing}, we get a vector bundle $ \bundle_X(E, z^kH) $ over $ X $. 
    Then, define $ \alpha_X([\bar{E}, \bar{F}, \bar{H}]) = \alpha_X([\proj_X^*E, \proj_X^*E, H]) \coloneqq [E^{\oplus k}] - [\bundle_X(E, z^k H)] \in K(X) $. 

    We will confirm the well-definedness of $ \alpha_X \colon K(X \times (D^2, S^1)) \to K(X) $. 
    First, in the case that we regard $ H $ as $ H = \{H_x(z) = \sum_{j = -k}^{l + 1} a_j(x)z^j \}_{x \in X} $ where $ a_{l + 1} = 0 $, its compatibility follows from the following commutativity:
    \[
        \begin{tikzcd}[column sep=huge]
            \Poly_{V, S^1}^{k + l, r} \arrow[d, hookrightarrow] \arrow{dr}{\conf_{\disk}} & \\
            \Poly_{V, S^1}^{k + l + 1, r} \arrow{r}[swap]{\conf_{\disk}} & \Conf_{V, \disk}^r
        \end{tikzcd}
    \]
    In the case that we regard $ H $ as $ H = \{H_x(z) = \sum_{j = -k - 1}^{l} a_j(x)z^j \}_{x \in X} $ where $ a_{k + 1} = 0 $, it is compatible because $ [E^{\oplus k + 1}] - [\bundle_X(E, z^{k + 1} H)] = [E^{\oplus k + 1}] - [\bundle_X(E, z^k H) \oplus E] = [E^{\oplus k}] - [\bundle_X(E, z^k H)] \in K(X) $ by \cref{lem:property-of-bundle} (4). 
    Next, it is independent of the choice of Cezaro means $ H_n $ of Fourier series of $ H $ if $ n $ is sufficiently large because we can connect $ H_n $ and $ H_{n + 1} $ linearly and construct a bundle in the same way as \cref{const:gluing} for a family of Laurent polynomials parametrized by $ X \times [0, 1] $. 
    From the above, the element of $ K(X) $ is defined, independent of the choice of $ H_n $ as a replacement of $ H $. 
    Finally, we confirm the compatibility regarding homotopy and stabilization. 
    About homotopy, considering that $ H $ is parametrized by $ X \times [0, 1] $ and constructing a bundle in the same way as the above, we can see that the same elements of $ K(X) $ are defined. 
    About stabilization, it is compatible by \cref{lem:property-of-bundle} (2) and (3). 
    Base on the above, $ \alpha_X \colon K(X \times (D^2, S^1)) \to K(X) $ is well-defined. 

    To show that $ \alpha_X \colon K(X \times (D^2, S^1)) \to K(X) $ is the inverse map of $ \beta_X \colon K(X) \to K(X \times (D^2, S^1)) $, it is enough to confirm the following three conditions (\cite{MR0228000}). 
    \begin{description}
        \item[(A1)] $ \alpha_X $ is functorial in $ X $. 
        \item[(A2)] $ \alpha_X $ is a $ K(X) $-module homomorphism. 
        \item[(A3)] $ \alpha_X([\underline{\Cbb}, \underline{\Cbb}, z^{-1}]) = [\underline{\Cbb}] $. 
    \end{description}

    First, we show (A1). 
    Let $ X $ and $ Y $ be compact Hausdorff spaces, $ f \colon Y \to X $ a continuous map. 
    It is enough to show that the following diagram is commutative and it follows from \cref{lem:property-of-bundle} (1). 
    \[
        \begin{tikzcd}[column sep=large]
            K(X \times (D^2, S^1)) \arrow{r} \arrow{d} & K(X) \arrow{d} \\
            K(Y \times (D^2, S^1)) \arrow{r} & K(Y)
        \end{tikzcd}
    \]

    Next, (A2) follows from \cref{lem:property-of-bundle} (2) and (5). 

    Finally, we show (A3).
    Define $ H \colon X \times S^1 \to \GL_1(\Cbb) = \Cbb^{\times} $ by $ H_x(z) \coloneqq z^{-1} $ ($ x \in X $, $ z \in S^1 $). 
    Then, $ zH = 1 \in \Poly_{\Cbb, S^1}^{0, 0} $ and $ \bundle_X(\underline{\Cbb}, zH) = 0 $ by \cref{lem:property-of-bundle} (3). 
    Therefore, $ \alpha_X([\underline{\Cbb}, \underline{\Cbb}, z^{-1}]) = [\underline{\Cbb}] - [0] = [\underline{\Cbb}] $. 
\end{proof}

\section{Comparison with other proofs of the Bott periodicity theorem}\label{sec:relation}

In this section, we will compare our proof with Atiyah--Bott's proof (\cite{MR0178470}) and Atiyah's proof (\cite{MR0228000}). 
Before this, we make one additional point regarding \cref{const:basic}. 
Suppose a compact Hausdorff space $ X $, a finite-dimensional vector space $ V $ and a continuous map $ H \colon X \to \Poly_{V, S^1}^{l, r} $ are given in the same as \cref{const:basic}. 
The composition map $ \conf(H) = \conf_{\disk} \circ H \colon X \to \Poly_{V, S^1}^{l, r} \to \Conf_{V, \disk}^r $ defines bundle morphisms
\begin{align*}
    A(H) \colon \bundle_X(V, H) \to \bundle_X(V, H), \quad
    \iota(H) \colon \underline{V} \to \bundle_X(V, H) 
\end{align*}
by \cref{lem:configuration-space-one-to-one}. 
Furthermore, a surjective morphism 
\[
    \pi(H) \coloneqq \sum_{i = 0}^{r - 1} A(H)^i \iota(H) \colon \underline{V}^{\oplus r} \to \bundle_X(V, H)
\]
is defined. 

\subsection{Comparison with Atiyah--Bott's proof}\label{subsec:atiyah-bott}

In the proof by Atiyah--Bott, clutching functions of vector bundles over $ S^2 $ are approximated by Laurent polynomials using Fourier series expansion, replaced by polynomials, replaced by linear polynomials, and finally they get vector bundles in the way of taking Bloch bundles. 
On the other hand, we give a way of directly getting vector bundles from polynomials. 
Furthermore, restricting our construction to the case of linear polynomials, by the following discussion, we can show that both constructions are isomorphic by the following discussion. 

First, we discuss the case limited to fibers. 
Let $ V $ be a $ d $-dimensional complex vector space. 
Suppose that $ A \in \End_{\Cbb}(V) $ does not have eigenvalues in $ S^1 $. 
Let a projection $ p $ on $ V $ be defined by 
\[
    p \coloneqq \frac{1}{2\pi i} \int_{S^1} (z \cdot \id - A)^{-1} dz. 
\]
On the other hand, let $ f(z) \coloneqq z - A \in \Poly_{V}^{1, inv} $. 
We define $ \tilde{f} \colon \Ocal_{\disk}(\textendash; V) \to \Ocal_{\disk}(\textendash; V) $ in the same way of \cref{subsec:configuration-map}. 

\begin{lemma}\label{lem:compare1}
    The composition $ p(V) \hookrightarrow V \hookrightarrow \Ocal_{\disk}(\disk; V) \to H^0(\Coker{\tilde{f}}) $ is an isomorphism. 
\end{lemma}

\begin{proof}
    Taking the Jordan normal form of $ A $ and corresponding direct sum decomposition of $ V $ by Jordan blocks, we can assume that $ A $ consists of one Jordan block
    \[
        \begin{pmatrix}
            \lambda & 1 & & & \\
            & \lambda & 1 & & \\
            & & \ddots & \ddots & \\
            & & & \ddots & 1 \\
            & & & & \lambda
        \end{pmatrix}. 
    \]
    Note that $ \abs{\lambda} \neq 1 $. 
    Then, it is enough to show the following two claims:
    \begin{enumerate}
        \item If $ \abs{\lambda} < 1 $, 
        \[
            V \to \Coker\left(\tilde{f}(\disk) = (z - A) \colon \Ocal_{\disk}(\disk; V) \to \Ocal_{\disk}(\disk; V)\right)
        \]
        is isomorphic and 
        \item if $ \abs{\lambda} > 1 $, 
        \[
            \Coker\left(\tilde{f}(\disk) = (z - A) \colon \Ocal_{\disk}(\disk; V) \to \Ocal_{\disk}(\disk; V)\right) = 0.
        \] 
    \end{enumerate}
    Since $ \tilde{f}(\disk) = (z - A) $ is invertible, (2) holds. 
    It is easy to see the injectivity of (1). 
    Result of elementary matrix operation as $ \Ocal_{\disk, \lambda} $-coefficient matrix, we have
    \[
        z - A = 
        \begin{pmatrix}
            z - \lambda & -1 & & & \\
            & z - \lambda & -1 & & \\
            & & \ddots & \ddots & \\
            & & & \ddots & -1 \\
            & & & & z - \lambda
        \end{pmatrix}
        \sim
        \begin{pmatrix}
            1 & & & & \\
            & 1 & & & \\
            & & \ddots & & \\
            & & & \ddots & \\
            & & & & (z - \lambda)^d
        \end{pmatrix}
    \]
    so $ \dim_{\Cbb}{\Coker\left(\tilde{f}(\disk) = (z - A) \colon \Ocal_{\disk}(\disk; V) \to \Ocal_{\disk}(\disk; V)\right)} = d $. 
    Thus, (1) is isomorphic. 
\end{proof}

Next, we consider the case parametrized by a space $ X $. 
Suppose a continuous map $ A \colon X \to \End_{\Cbb}(V) $ is given. 
For each $ x \in X $, we assume that $ A(x) $ does not have eigenvalues in $ S^1 $ and the number of eigenvalues with multiplicity in $ \disk $ is $ r $. 
We define a family of projection on $ V $ $ p \colon \underline{V} \to \underline{V} $ as
\[
    p_x \coloneqq \frac{1}{2\pi i} \int_{S^1} (z \cdot \id - A(x))^{-1} dz.
\]
On the other hand, let $ f \colon X \to \Poly_{V}^{1, inv} $ be defined by $ f_x(z) = z - A(x) $ for $ x \in X $. 
The following proposition follows from \cref{lem:compare1}. 

\begin{proposition}
    The composition of bundle morphisms over $ X $
    \[
        p(\underline{V}) \hookrightarrow \underline{V} \hookrightarrow \underline{V}^{\oplus r} \xrightarrow{\pi(f)} \bundle_X(V, f)
    \]
    is isomorphic. 
    Here, the second map is inclusion into the first component. 
\end{proposition}

\subsection{Comparison with Atiyah's proof}\label{subsec:atiyah}

In the proof by Atiyah, they take the Toeplitz operators of clutching functions and get a element of $ K $-group as their family indices. 
In our construction, we define the indices using cokernels of sheaf homomorphisms, limiting the case that clutching functions are given as polynomials. 
This clarifies that there is a preliminary step in the construction of obtaining vector bundles, namely, obtaining families of configurations labeled by vector spaces. 

First, we discuss the case limited to fibers. 
Let $ V $ be a finite-dimensional vector space. 
Let $ f(z) = \sum_{j = 0}^l a_j z^j \in \Poly_{V, S^1}^{l, r} $ using the notation in \cref{subsec:configuration-map}. 
Define $ \tilde{f} \colon \Ocal_{\disk}(\textendash; V) \to \Ocal_{\disk}(\textendash; V) $ in the same way of \cref{subsec:configuration-map}. 
Also, define $ \bar{f} \colon l^2(\Zbb_{\geq 0}; V) \to l^2(\Zbb_{\geq 0}; V) $ as $ \bar{f} \coloneqq \sum_{j = 0}^l a_j S^j $. 
Here, $ S \colon l^2(\Zbb_{\geq 0}; V) \to l^2(\Zbb_{\geq 0}; V) $ is a shift operator satisfying $ S((v_i)_i) = (v_{i - 1})_i $. 

\begin{definition}[Hardy space]
    For $ \varphi(z) \in \Ocal_{\disk}(\disk; V) $ and $ 0 \leq r < 1 $, let $ \varphi_r \colon S^1 \to V $ be defined by $ \varphi_r(z) \coloneqq \varphi(rz) $ ($ z \in S^1 $).
    Define norm of $ \varphi $ as $ \norm{\varphi}_h \coloneqq \limsup_{r \nearrow 1} \norm{\varphi_r}_{L^2} $. 
    Then, we define a subspace of $ \Ocal_{\disk}(\disk; V) $ by 
    \[
      \Ocal_{\disk}(\disk; V)_{L^2} \coloneqq \{ \varphi \in \Ocal_{\disk}(\disk; V) \mid \norm{\varphi}_h < \infty \}. 
    \]

    If a linear map $ \iota \colon l^2(\Zbb_{\geq 0}; V) \to \Ocal_{\disk}(\disk; V)_{L^2} $ is defined by $ \iota((v_i)_i) = \sum_{i \geq 0} v_i z^i $, this map is well-defined and $ \Cbb $-isomorphic and satisfies that for any $ v \in l^2(\Zbb_{\geq 0}; V) $, $ \norm{v}_{l^2} = \norm{\iota(v)}_h $. 
    In paticular, the space $ \Ocal_{\disk}(\disk; V)_{L^2} $ becomes a Hilbert space by inner product $ (\varphi, \psi) = \lim_{r \nearrow 1} (\varphi_r, \psi_r) $. 
    This space is so-called Hardy space. 
\end{definition}

Let 
\[
    \tilde{f}(\disk)_{L^2} \colon \Ocal_{\disk}(\disk; V)_{L^2} \to \Ocal_{\disk}(\disk; V)_{L^2}
\]
be the restriction of $ \tilde{f}(\disk) \colon \Ocal_{\disk}(\disk; V) \to \Ocal_{\disk}(\disk; V) $ to $ \Ocal_{\disk}^{L^2}(\disk; V) $. 
Then, we obtain the following morphisms between short exact sequences:

\[
    \begin{tikzcd}
        0 \arrow{r} 
        & l^2(\Zbb_{\geq 0}; V) \arrow{r}{\bar{f}} \arrow[d,"\iota","\cong"']
        & l^2(\Zbb_{\geq 0}; V) \arrow{r} \arrow[d,"\iota","\cong"']
        & \Coker{\bar{f}} \arrow{r} \arrow[d,"\cong"']
        & 0  \\
        0 \arrow{r} 
        & \Ocal_{\disk}(\disk; V)_{L^2} \arrow{r}{\tilde{f}(\disk)_{L^2}} \arrow[d,hookrightarrow]
        & \Ocal_{\disk}(\disk; V)_{L^2} \arrow{r} \arrow[d,hookrightarrow]
        & \Coker{\tilde{f}(\disk)_{L^2}} \arrow{r} \arrow{d}
        & 0 \\
        0 \arrow{r} 
        & \Ocal_{\disk}(\disk; V) \arrow{r}{\tilde{f}(\disk)} 
        & \Ocal_{\disk}(\disk; V) \arrow{r} 
        & \Coker{\tilde{f}(\disk)} \arrow{r} 
        & 0
    \end{tikzcd}
\]

\begin{lemma}\label{lem:compare2}
    The composition
    \[
        \Coker{\bar{f}} \cong \Coker{\tilde{f}(\disk)_{L^2}} \to \Coker{\tilde{f}(\disk)} \cong H^0(\Coker{\tilde{f}})
    \]
    is an isomorphism. 
\end{lemma}

\begin{proof}
    We use the notation in the proof of \cref{lem:support} and \cref{lem:surj}. 
    However, $ \Zcal $ denotes the set of roots of $ \det{f(z)} $ that are in $ \disk $. 
    Same as \cref{lem:surj}, the strategy is to take pullback of $ \Ocal_{\disk}(\disk; V) $ and $ \Ocal_{\disk}(\disk; V)_{L^2} $ along $ \tilde{f}^m(\disk) $ and consider them in $ \Mcal_{\disk}(\disk; V) $. 
    Let $ \Mcal_{\disk}(\disk; V)^{\tilde{f}} \subset \Mcal_{\disk}(\disk; V) $ be defined in the same manner as \cref{lem:surj}, and
    \[
        \Mcal_{\disk}(\disk; V)_{L^2}^{\tilde{f}} \coloneqq \{\varphi \in \Mcal_{\disk}(\disk; V)^{\tilde{f}} \mid \norm{\varphi}_h < \infty\}. 
    \]
    (The poles of $ \varphi \in \Mcal_{\disk}(\disk; V)^{\tilde{f}} $ is finite so $ \norm{\varphi}_h $ is defined in the same way as $ \varphi \in \Ocal_{\disk}(\disk; V) $. )
    Then, we have $ \Mcal_{\disk}(\disk; V)_{L^2}^{\tilde{f}} = \tilde{f}^m(\disk)^{-1}(\Ocal_{\disk}(\disk; V)_{L^2}) $. 
    The injectivity and surjectivity to be shown can be rephrased as the following claims, respectively:
    \begin{align*}
        \Ocal_{\disk}(\disk; V)_{L^2} & = \Ocal_{\disk}(\disk; V) \cap \Mcal_{\disk}(\disk; V)_{L^2}^{\tilde{f}}, \\
        \Mcal_{\disk}(\disk; V)^{\tilde{f}} & = \Ocal_{\disk}(\disk; V) + \Mcal_{\disk}(\disk; V)_{L^2}^{\tilde{f}}. 
    \end{align*}
    The first claim clearly holds, and the second claim follows because it is holomorphic if principal part is subtracted.
\end{proof}

Next, we consider the case parametrized by a space $ X $. 
Suppose a continuous map $ f \colon X \to \Poly_{V, S^1}^{l, r} $ is given. 
Define a homomorphism of Hilbert bundle over $ X $ $ \bar{f} \colon \underline{l^2(\Zbb_{\geq 0}; V)} \to \underline{l^2(\Zbb_{\geq 0}; V)} $ as in the case of fibers. 

\begin{proposition}
    The cokernel bundle $ \Coker{\bar{f}} $ and $ \bundle_X(V, f) $ are naturally isomorphic as vector bundles over $ X $. 
\end{proposition}

\begin{proof}
    By \cref{lem:compare2}, for any $ x \in X $, $ (\Coker{\bar{f}})_x \cong \bundle_X(V, f)_x $. 
    Since 
    \[
        \pi(f) \colon \underline{V}^{\oplus r} \to \bundle_X(V, f)
    \]
    is surjective, the composition 
    \[
        \underline{V}^{\oplus r} \hookrightarrow \underline{l^2(\Zbb_{\geq 0}; V)} \to \Coker{\bar{f}}
    \]
    is also surjective. 
    Here, the first inclusion is defined by 
    \[
        V^{\oplus r} \to l^2(\Zbb_{\geq 0}; V) \ ; (v_i)_{i = 0, \dots, r - 1} \mapsto (v_0, \dots, v_{r - 1}, 0, 0, \dots). 
    \]
    Two maps $ \pi(f) \colon \underline{V}^{\oplus r} \to \bundle_X(V, f) $ and $ \underline{V}^{\oplus r} \to \Coker{\bar{f}} $ have same kernels so we have naturally $ \bundle_X(V, f) \cong \Coker{\bar{f}} $. 
\end{proof}

\part{An alternative proof of the bulk-edge correspondence}\label{part:bulkedge}

\section{A toy model for the proof of the bulk-edge correspondence}\label{sec:toymodel}

In this section, we describe the low-dimensional version of the bulk-edge correspondence in \cref{sec:setting} and \ref{sec:proof-of-bulkedge}. 

Let $ V $ be a finite-dimensional complex inner product space. 
Let $ D^1 \coloneqq [-1, 1] $ and $ S^0 \coloneqq \partial D^1 $. 
Let $ \Herm(V) \subset \End(V) $ be the set of Hermite matrices on $ V $. 
Also, let $ \Herm_0(V) \subset \Herm(V) $ be the set of Hermite matrices that do not have 0 as eigenvalue. 
Let $ A \colon (D^1, S^0) \to (\Herm(V), \Herm_0(V)) $ be a continuous path. 
Then, spectral flow $ \sf(A) \in Z $ of $ A $ is defined. 
For example, see \cite{MR1426691}. 
Roughly speaking, spectral flow $ \sf(A) $ is a signed count of the number of times the left-to-right flow of eigenvalues intersects 0.

\begin{example}\label{ex:spectral-flow}
    Let $ A \colon (D^1, S^0) \to (\Herm(\Cbb^3), \Herm_0(\Cbb^3)) $ be defined by 
    \[
        A_t \coloneqq 
        \begin{pmatrix}
            \frac{1}{2}t + \frac{1}{4} & & \\
            & - \frac{1}{2}t & \\
            & & \frac{1}{4}t^2 + \frac{1}{2} t - \frac{1}{4}
        \end{pmatrix}
    \]
    for $ -1 \leq t \leq 1 $. 
    Then, spectral flow $ \sf(A) $ of $ A $ is 1. 
    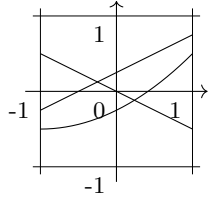
\begin{figure}[ht]
        \centering
        \begin{tikzpicture}[domain=-1:1]
            \draw[very thin] (1.0, -1.1) -- (1.0, 1.1);
            \draw[very thin] (-1.0, -1.1) -- (-1.0, 1.1);
            \draw[very thin] (-1.1, 1.0) -- (1.1, 1.0);
            \draw[very thin] (-1.1, -1.0) -- (1.1, -1.0);
            \draw[->] (-1.2, 0) -- (1.2, 0);
            \draw[->] (0, -1.2) -- (0, 1.2);
            \draw plot (\x, \x * 0.5 + 0.25);
            \draw plot (\x, - \x * 0.5);
            \draw plot (\x, \x * \x * 0.25 + \x * 0.5 - 0.25);
            \draw (0,0) node[below left] {{\scriptsize 0}};
            \draw (0,1) node[below left] {{\scriptsize 1}};
            \draw (0,-1) node[below left] {{\scriptsize -1}};
            \draw (1,0) node[below left] {{\scriptsize 1}};
            \draw (-1,0) node[below left] {{\scriptsize -1}};
        \end{tikzpicture}
        \caption{The graph of eigenvalues of $ \{A_t\}_t $ in \cref{ex:spectral-flow}.}\label{fig:tikz-graph}
    \end{figure}
\end{example}

Since the operators we are dealing with are finite-dimensional, we immediately know the following. 

\begin{lemma}
    Let $ A \colon (D^1, S^0) \to (\Herm(V), \Herm_0(V)) $ be a continuous path. 
    Then, we have
    \[
        \sf(A) = \rank_{\Cbb}(\chi_{[0, \infty)}(A_1)) - \rank_{\Cbb}(\chi_{[0, \infty)}(A_{-1})). 
    \]
    Here, $ \chi_{[0, \infty)} $ is the characteristic function on $ [0, \infty) $. 
\end{lemma}

In the following, we describe a low-dimensional version of the proof of the bulk-edge correspondence. 
Let $ H \colon S^0 \times S^0 \to \Herm_0(V) $ be a continuous map. 
The map $ H $ corresponds to the continuous family of invertible matrices
\[
    \{\hat{H}_x(z) - \lambda \in \GL_{\Cbb}(E_x)\}_{x \in X, z \in S^1, \lambda \in \gamma}
\]
in the proof of \cref{thm:bulkedge}. 
Let $ H_1 \colon S^0 \times D^1 \to \Herm(V) $ and $ H_2 \colon D^1 \times S^0 \to \Herm(V) $ are extentions of $ H $, respectively. 

\begin{definition}
    Let the first index $ \ind_1(H) $ be spectral flow of $ H_1 $: 
    \[
        \ind_1(H) \coloneqq \sf(H_1) = \sf(\rest{H_1}{\{1\} \times D^1}) - \sf(\rest{H_1}{\{-1\} \times D^1})
    \]
    Also, let the second index $ \ind_2(H) $ be spectral flow of $ H_2 $: 
    \[
        \ind_2(H) \coloneqq \sf(H_2) = \sf(\rest{H_2}{D^1 \times \{1\}}) - \sf(\rest{H_2}{D^1 \times \{-1\}})
    \]
\end{definition}

The first index $ \ind_1(H) $ corresponds to the bulk index $ \ind^b(E, H) $ in \cref{sec:setting}, and the second index $ \ind_2(H) $ corresponds to the edge index $ \ind^e(E, H) $. 

\begin{figure}[ht]
    \centering
    \tikzset{every picture/.style={line width=0.75pt}} 

\begin{tikzpicture}[x=0.75pt,y=0.75pt,yscale=-1,xscale=1]

\draw    (60,142) -- (239,142) ;
\draw [shift={(241,142)}, rotate = 180] [color={rgb, 255:red, 0; green, 0; blue, 0 }  ][line width=0.75]    (10.93,-3.29) .. controls (6.95,-1.4) and (3.31,-0.3) .. (0,0) .. controls (3.31,0.3) and (6.95,1.4) .. (10.93,3.29)   ;
\draw    (150,224) -- (150,60) ;
\draw [shift={(150,58)}, rotate = 90] [color={rgb, 255:red, 0; green, 0; blue, 0 }  ][line width=0.75]    (10.93,-3.29) .. controls (6.95,-1.4) and (3.31,-0.3) .. (0,0) .. controls (3.31,0.3) and (6.95,1.4) .. (10.93,3.29)   ;
\draw    (103,174) -- (196.34,111.12) ;
\draw [shift={(198,110)}, rotate = 146.03] [color={rgb, 255:red, 0; green, 0; blue, 0 }  ][line width=0.75]    (10.93,-3.29) .. controls (6.95,-1.4) and (3.31,-0.3) .. (0,0) .. controls (3.31,0.3) and (6.95,1.4) .. (10.93,3.29)   ;

\draw  [dash pattern={on 3pt off 3pt}]  (120,132) -- (212,132) ;
\draw  [dash pattern={on 3pt off 3pt}]  (91,152) -- (183,152) ;
\draw  [dash pattern={on 3pt off 3pt}]  (91,152) -- (120,132) ;
\draw  [dash pattern={on 3pt off 3pt}]  (183,152) -- (212,132) ;

\draw  [dash pattern={on 0.84pt off 2.51pt}]  (91,82) -- (91,223) ;
\draw  [dash pattern={on 0.84pt off 2.51pt}]  (121,62) -- (121,203) ;
\draw  [dash pattern={on 0.84pt off 2.51pt}]  (211,62) -- (211,203) ;
\draw  [dash pattern={on 0.84pt off 2.51pt}]  (183,82) -- (183,223) ;

\draw [fill] (91,119)  circle [radius=3];
\draw [fill] (91,192)  circle [radius=3];
\draw [fill] (121,106) circle [radius=3];
\draw [fill] (121,83)  circle [radius=3];
\draw [fill] (211,166) circle [radius=3];
\draw [fill] (211,112) circle [radius=3];
\draw [fill] (183,106) circle [radius=3];
\draw [fill] (183,207) circle [radius=3];

\draw (194,146) node [anchor=north west][inner sep=0.75pt]   [align=left] {1};
\draw (158,116) node [anchor=north west][inner sep=0.75pt]   [align=left] {1};
\draw (133,155) node [anchor=north west][inner sep=0.75pt]   [align=left] {-1};
\draw (93,126) node [anchor=north west][inner sep=0.75pt]   [align=left] {-1};
\draw (154,61) node [anchor=north west][inner sep=0.75pt]   [align=left] {$ \Rbb $};

\end{tikzpicture}
    \caption{An example of spectrum of $ H $. }
\end{figure}
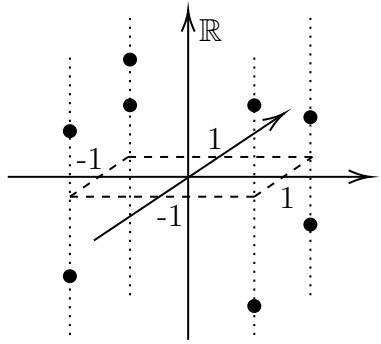

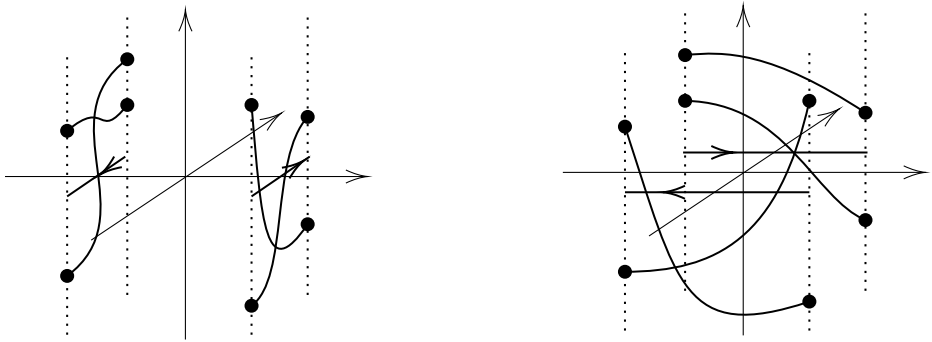
\begin{figure}[ht]
    \begin{minipage}{0.48\columnwidth}
        \centering
        \tikzset{every picture/.style={line width=0.75pt}} 

\begin{tikzpicture}[x=0.75pt,y=0.75pt,yscale=-1,xscale=1]

\draw  [thin]  (60,142) -- (239,142) ;
\draw [shift={(241,142)}, rotate = 180] [color={rgb, 255:red, 0; green, 0; blue, 0 }  ][thin]    (10.93,-3.29) .. controls (6.95,-1.4) and (3.31,-0.3) .. (0,0) .. controls (3.31,0.3) and (6.95,1.4) .. (10.93,3.29)   ;
\draw  [thin]  (150,224) -- (150,60) ;
\draw [shift={(150,58)}, rotate = 90] [color={rgb, 255:red, 0; green, 0; blue, 0 }  ][thin]    (10.93,-3.29) .. controls (6.95,-1.4) and (3.31,-0.3) .. (0,0) .. controls (3.31,0.3) and (6.95,1.4) .. (10.93,3.29)   ;
\draw  [thin]  (103,174) -- (196.34,111.12) ;
\draw [shift={(198,110)}, rotate = 146.03] [color={rgb, 255:red, 0; green, 0; blue, 0 }  ][thin]    (10.93,-3.29) .. controls (6.95,-1.4) and (3.31,-0.3) .. (0,0) .. controls (3.31,0.3) and (6.95,1.4) .. (10.93,3.29)   ;

\draw    (91,152) -- (120,132) ;
\draw    (183,152) -- (212,132) ;

\draw  [dash pattern={on 0.84pt off 2.51pt}]  (91,82) -- (91,223) ;
\draw  [dash pattern={on 0.84pt off 2.51pt}]  (121,62) -- (121,203) ;
\draw  [dash pattern={on 0.84pt off 2.51pt}]  (211,62) -- (211,203) ;
\draw  [dash pattern={on 0.84pt off 2.51pt}]  (183,82) -- (183,223) ;

\draw [fill] (91,119)  circle [radius=3];
\draw [fill] (91,192)  circle [radius=3];
\draw [fill] (121,106) circle [radius=3];
\draw [fill] (121,83)  circle [radius=3];
\draw [fill] (211,166) circle [radius=3];
\draw [fill] (211,112) circle [radius=3];
\draw [fill] (183,106) circle [radius=3];
\draw [fill] (183,207) circle [radius=3];

\draw    (91,119) .. controls (111,102) and (106,125) .. (121,106) ;
\draw    (91,192) .. controls (131,162) and (81,113) .. (121,83) ;
\draw    (183,106) .. controls (188,167) and (190,197) .. (211,166) ;
\draw    (211,112) .. controls (192,133) and (202,193) .. (183,207) ;

\draw [shift={(107.5,141)}, rotate = 324.25] [color={rgb, 255:red, 0; green, 0; blue, 0 }  ][line width=0.75]    (10.93,-3.29) .. controls (6.95,-1.4) and (3.31,-0.3) .. (0,0) .. controls (3.31,0.3) and (6.95,1.4) .. (10.93,3.29)   ;
\draw [shift={(209,134)}, rotate = 145.01] [color={rgb, 255:red, 0; green, 0; blue, 0 }  ][line width=0.75]    (10.93,-3.29) .. controls (6.95,-1.4) and (3.31,-0.3) .. (0,0) .. controls (3.31,0.3) and (6.95,1.4) .. (10.93,3.29)   ;

\end{tikzpicture}
    \end{minipage}
    \begin{minipage}{0.48\columnwidth}
        \centering
        \tikzset{every picture/.style={line width=0.75pt}} 

\begin{tikzpicture}[x=0.75pt,y=0.75pt,yscale=-1,xscale=1]

\draw  [thin]  (60,142) -- (239,142) ;
\draw [shift={(241,142)}, rotate = 180] [color={rgb, 255:red, 0; green, 0; blue, 0 }  ][thin]    (10.93,-3.29) .. controls (6.95,-1.4) and (3.31,-0.3) .. (0,0) .. controls (3.31,0.3) and (6.95,1.4) .. (10.93,3.29)   ;
\draw  [thin]  (150,224) -- (150,60) ;
\draw [shift={(150,58)}, rotate = 90] [color={rgb, 255:red, 0; green, 0; blue, 0 }  ][thin]    (10.93,-3.29) .. controls (6.95,-1.4) and (3.31,-0.3) .. (0,0) .. controls (3.31,0.3) and (6.95,1.4) .. (10.93,3.29)   ;
\draw  [thin]  (103,174) -- (196.34,111.12) ;
\draw [shift={(198,110)}, rotate = 146.03] [color={rgb, 255:red, 0; green, 0; blue, 0 }  ][thin]    (10.93,-3.29) .. controls (6.95,-1.4) and (3.31,-0.3) .. (0,0) .. controls (3.31,0.3) and (6.95,1.4) .. (10.93,3.29)   ;

\draw    (120,132) -- (212,132) ;
\draw    (91,152) -- (183,152) ;

\draw  [dash pattern={on 0.84pt off 2.51pt}]  (91,82) -- (91,223) ;
\draw  [dash pattern={on 0.84pt off 2.51pt}]  (121,62) -- (121,203) ;
\draw  [dash pattern={on 0.84pt off 2.51pt}]  (211,62) -- (211,203) ;
\draw  [dash pattern={on 0.84pt off 2.51pt}]  (183,82) -- (183,223) ;

\draw [fill] (91,119)  circle [radius=3];
\draw [fill] (91,192)  circle [radius=3];
\draw [fill] (121,106) circle [radius=3];
\draw [fill] (121,83)  circle [radius=3];
\draw [fill] (211,166) circle [radius=3];
\draw [fill] (211,112) circle [radius=3];
\draw [fill] (183,106) circle [radius=3];
\draw [fill] (183,207) circle [radius=3];

\draw    (121,106) .. controls (172,107) and (185,155) .. (211,166) ;
\draw    (121,83) .. controls (141,81) and (164,81) .. (211,112) ;
\draw    (91,119) .. controls (117,199) and (120,228) .. (183,207) ;
\draw    (91,192) .. controls (137,192) and (168,175) .. (183,106) ;

\draw [shift={(145,132)}, rotate = 180] [color={rgb, 255:red, 0; green, 0; blue, 0 }  ][line width=0.75]    (10.93,-3.29) .. controls (6.95,-1.4) and (3.31,-0.3) .. (0,0) .. controls (3.31,0.3) and (6.95,1.4) .. (10.93,3.29)   ;
\draw [shift={(110,152)}, rotate = 360] [color={rgb, 255:red, 0; green, 0; blue, 0 }  ][line width=0.75]    (10.93,-3.29) .. controls (6.95,-1.4) and (3.31,-0.3) .. (0,0) .. controls (3.31,0.3) and (6.95,1.4) .. (10.93,3.29)   ;

\end{tikzpicture}
    \end{minipage}
    \caption{An example of spectrum of $ H_1 $ with $ \ind_1(H) = -1 $ and an example of spectrum of $ H_2 $ with $ \ind_2(H) = -1 $. }
\end{figure}

Then, corresponding to the bulk-edge correspondence (\cref{thm:bulkedge}), the following holds. 

\begin{proposition}
    $ \ind_1(H) = \ind_2(H) $. 
\end{proposition}

\begin{proof}
    Gluing $ H_1 \colon S^0 \times D^1 \to \Herm(V) $ and $ H_2 \colon D^1 \times S^0 \to \Herm(V) $, we obtain
    \[
        \bar{H} \colon S^1 \approx \partial(D^1 \times D^1) \to \Herm(V). 
    \]
    Therefore, 
    \[
        \ind_1(H) - \ind_2(H) = \sf(H_1) - \sf(H_2) = \sf(\bar{H}) = 0. 
    \]
\end{proof}

\begin{figure}[ht]
    \centering
    \tikzset{every picture/.style={line width=0.75pt}} 

\begin{tikzpicture}[x=0.75pt,y=0.75pt,yscale=-1,xscale=1]

\draw  [thin]  (60,142) -- (239,142) ;
\draw [shift={(241,142)}, rotate = 180] [color={rgb, 255:red, 0; green, 0; blue, 0 }  ][thin]    (10.93,-3.29) .. controls (6.95,-1.4) and (3.31,-0.3) .. (0,0) .. controls (3.31,0.3) and (6.95,1.4) .. (10.93,3.29)   ;
\draw  [thin]  (150,224) -- (150,60) ;
\draw [shift={(150,58)}, rotate = 90] [color={rgb, 255:red, 0; green, 0; blue, 0 }  ][thin]    (10.93,-3.29) .. controls (6.95,-1.4) and (3.31,-0.3) .. (0,0) .. controls (3.31,0.3) and (6.95,1.4) .. (10.93,3.29)   ;
\draw  [thin]  (103,174) -- (196.34,111.12) ;
\draw [shift={(198,110)}, rotate = 146.03] [color={rgb, 255:red, 0; green, 0; blue, 0 }  ][thin]    (10.93,-3.29) .. controls (6.95,-1.4) and (3.31,-0.3) .. (0,0) .. controls (3.31,0.3) and (6.95,1.4) .. (10.93,3.29)   ;

\draw  [thin]  (120,132) -- (212,132) ;
\draw  [thin]  (91,152) -- (183,152) ;
\draw  [thin]  (91,152) -- (120,132) ;
\draw  [thin]  (183,152) -- (212,132) ;

\draw  [dash pattern={on 0.84pt off 2.51pt}]  (91,82) -- (91,223) ;
\draw  [dash pattern={on 0.84pt off 2.51pt}]  (121,62) -- (121,203) ;
\draw  [dash pattern={on 0.84pt off 2.51pt}]  (211,62) -- (211,203) ;
\draw  [dash pattern={on 0.84pt off 2.51pt}]  (183,82) -- (183,223) ;

\draw [fill] (91,119)  circle [radius=3];
\draw [fill] (91,192)  circle [radius=3];
\draw [fill] (121,106) circle [radius=3];
\draw [fill] (121,83)  circle [radius=3];
\draw [fill] (211,166) circle [radius=3];
\draw [fill] (211,112) circle [radius=3];
\draw [fill] (183,106) circle [radius=3];
\draw [fill] (183,207) circle [radius=3];

\draw    (91,119) .. controls (111,102) and (106,125) .. (121,106) ;
\draw    (91,192) .. controls (131,162) and (81,113) .. (121,83) ;
\draw    (183,106) .. controls (188,167) and (190,197) .. (211,166) ;
\draw    (211,112) .. controls (192,133) and (202,193) .. (183,207) ;
\draw    (121,106) .. controls (172,107) and (185,155) .. (211,166) ;
\draw    (121,83) .. controls (141,81) and (164,81) .. (211,112) ;
\draw    (91,119) .. controls (117,199) and (120,228) .. (183,207) ;
\draw    (91,192) .. controls (137,192) and (168,175) .. (183,106) ;

\draw [shift={(138,132)}, rotate = 360] [color={rgb, 255:red, 0; green, 0; blue, 0 }  ][thin]    (10.93,-3.29) .. controls (6.95,-1.4) and (3.31,-0.3) .. (0,0) .. controls (3.31,0.3) and (6.95,1.4) .. (10.93,3.29)   ;

\end{tikzpicture}
    \caption{An example of spectrum of $ \bar{H} $. }
\end{figure}

\section{Setup of the bulk-edge correspondence and definition of the indices}\label{sec:setting}

Let $ X $ be a oriented smooth $ n $-dimensional closed manifold, $ E $ a complex vector bundle of rank $ d $ over $ X $. 
Suppose bounded linear operator 
\[
    H_x \colon l^2(\Zbb; E_x) \to l^2(\Zbb; E_x)
\]
is given for each $ x \in X $.
Write $ H \coloneqq \{H_x\}_{x \in X} $. 
We assume that $ H_x $ can be written as $ H_x = \sum_{j = -k}^l a_j(x) S^j $ using $ k, l \in \Zbb_{\geq 0} $, independent $ x \in X $. 
Here, $ S \colon l^2(\Zbb; E_x) \to l^2(\Zbb; E_x) $ is a shift operator satisfying $ S((v_i)_i) = (v_{i - 1})_i $. 
Also, for each $ -k \leq j \leq l $, $ a_j $ is a continuous section $ a_j \colon X \to \End(E) $ of complex vector bundle $ \End(E) \to X $. 
Furthermore, we suppose that a simple closed curve $ \gamma $ in $ \Cbb $ which satisfies that $ \sigma(H_x) \cap \gamma = \emptyset $ where $ \sigma(H_x) $ is the spectrum of $ H_x $ for each $ x \in X $ is given. 

\subsection{Definition of bulk index}

We consider $ \hat{H}_x \colon L^2(S^1; E_x) \to L^2(S^1; E_x) $ corresponding $ H_x $ by Fourier transformation. 
Let $ z $ be a coordinate of $ S^1 $, then we can write $ \hat{H}_x = \sum_{j = -k}^l a_j(x) z^j $. 
We define a subbundle $ E_{\mathrm{bulk}} $ of a complex vector bundle $ \proj_{X}^* E \to X \times S^1 $ as follows:
Considering a projection 
\[
    \frac{1}{2\pi i} \int_{\gamma}(\lambda\cdot\id - \hat{H}_x(z))^{-1}d\lambda
\]
on $ (\proj_{X}^* E)_{x, z} = E_x $ for each $ x \in X $ and $ z \in S^1 $, let $ E_{\mathrm{bulk}} $ be the subbundle of $ \proj_{X}^* E $ defined as their images. 
The fiber of $ E_{\mathrm{bulk}} $ at $ (x, z) \in X \times S^1 $ coincides with the direct sum of generalized eigenspaces corresponding to eigenvalues of $ \hat{H}_x(z) $ inside $ \gamma $. 

\begin{definition}
    We define the bulk index $ \mathrm{ind}^{b}(E, H) $ of $ \{H_x\}_{x\in X} $ as 
    \[
        \mathrm{ind}^b(E, H) \coloneqq - \int_{X \times S^1} ch([E_{\mathrm{bulk}}]).
    \]
\end{definition}

\subsection{Definition of edge index}

Let $ P_{\geq 0} \colon l^2(\Zbb; E_x) \to l^2(\Zbb_{\geq 0}; E_x) $ be the projection. 
The adjoint operator $ P_{\geq 0}^* \colon l^2(\Zbb_{\geq 0}; E_x) \to l^2(\Zbb; E_x) $ of $ P_{\geq 0} $ is given by 
\[
    P_{\geq 0}^*((v_i)_i) \coloneqq (w_i)_i, \quad \text{where} \quad w_i =
    \begin{cases}
        0 & \text{if } i < 0 \\
        v_i & \text{if } i \geq 0 
    \end{cases}. 
\]
For each $ x \in X $, $ \lambda \in \gamma $, we define 
\[
    (H_x - \lambda)^\# \coloneqq P_{\geq 0}(H_x - \lambda)P_{\geq 0}^* \colon l^2(\Zbb_{\geq 0}; E_x) \to l^2(\Zbb_{\geq 0}; E_x). 
\]

\begin{lemma}
    A family $ \{(H_x - \lambda)^\#\}_{x, \lambda} $ is a continuous family of Fredholm operators.  
\end{lemma}

\begin{proof}
    For $ x \in X $ and $ \lambda \in \gamma $, $ H_x - \lambda $ is invertible, in paticular, it is Fredholm. 
    The difference 
    \[
        (H_x - \lambda) - (P_{\geq 0}(H_x - \lambda)P_{\geq 0}^* \oplus P_{< 0}(H_x - \lambda)P_{< 0}^*)
    \]
    is finite rank operator, in paticular, 
    \[
        P_{\geq 0}(H_x - \lambda)P_{\geq 0}^* \oplus P_{< 0}(H_x - \lambda)P_{< 0}^*
    \]
    is Fredholm, and $ (H_x - \lambda)^\# = P_{\geq 0}(H_x - \lambda)P_{\geq 0}^* $ is also Fredholm because it is a direct summand. 
\end{proof}

\begin{definition}
    Let $ \mathrm{ind}H^\# \coloneqq \operatorname{ind}(\{(H_x - \lambda)^\#\}_{x \in X, \lambda \in \gamma}) \in K(X \times \gamma) $. 
    We define the edge index $ \mathrm{ind}^e(E, H) $ of $ \{H_x\}_{x\in X} $ as 
    \[
        \mathrm{ind}^e(E, H) \coloneqq - \int_{X \times \gamma} ch(\mathrm{ind}H^\#). 
    \]
\end{definition}

\section{Proof of the bulk-edge correspondence}\label{sec:proof-of-bulkedge}

\begin{theorem}\label{thm:bulkedge}
    Under the setup in \cref{sec:setting}, 
    \[
        \mathrm{ind}^b(E, H) = \mathrm{ind}^e(E, H). 
    \]
\end{theorem}

\begin{proof}
    Let $ \bar{D_\gamma} $ be union of the bounded one of two connected components of $ \Cbb \setminus \gamma $ and $ \gamma $. 
    Let
    \begin{align*}
        \beta_1 &\colon K(X \times S^1) \to K(X \times S^1 \times (\bar{D_\gamma}, \gamma)) \text{ and} \\
        \beta_2 &\colon K(X \times \gamma) \to K(X \times (D^2, S^1) \times \gamma)
    \end{align*}
    be the Bott map (\cref{sec:strategy}), respectively. 
    Replacing $ (\bar{D_\gamma}, \gamma) $ with $ (D^2, S^1) $ by a biholomorphic map, $ \beta_1 $ is defined. 
    For each $ x \in X $, $ z \in S^1 $ and $ \lambda \in \gamma $, $ \hat{H}_x(z) - \lambda \in \GL_{\Cbb}(E_x) $. 
    Thus, 
    \begin{align*}
        [\proj_X^* E, \proj_X^* E, \{\hat{H}_x(z) - \lambda\}_{x, z, \lambda}] &\in K(X \times S^1 \times (\bar{D_\gamma}, \gamma)) \text{ and} \\
        [\proj_X^* E, \proj_X^* E, \{\hat{H}_x(z) - \lambda\}_{x, z, \lambda}] &\in K(X \times (D^2, S^1) \times \gamma). 
    \end{align*}
    Let $ u $ and $ v $ denote them, respectively. 
    By \cref{subsec:atiyah-bott}, \ref{subsec:atiyah}, we have
    \begin{align*}
        \beta_1^{-1}(u) &= - [E_{\mathrm{bulk}}] \in K(X \times S^1) \text{ and} \\
        \beta_2^{-1}(v) &= \mathrm{ind}H^\# \in K(X \times \gamma)
    \end{align*}
    Thus, we have
    \begin{align*}
        \mathrm{ind}^b(E, H) - \mathrm{ind}^e(E, H) &= - \int_{X \times S^1} ch([E_{\mathrm{bulk}}]) + \int_{X \times \gamma} ch(\mathrm{ind}H^\#) \\
        &= \int_X \left(\int_{S^1} ch(\beta_1^{-1}(u)) + \int_{\gamma} ch(\beta_2^{-1}(v))\right) \\
        &= \int_X \left(\int_{S^1} \int_{\bar{D_\gamma}} ch(u) + \int_{\gamma} \int_{D^2} ch(v)\right). 
    \end{align*}
    Let $ T \coloneqq S^1 \times \bar{D_\gamma} \cup_{S^1 \times \gamma} D^2 \times \gamma = \partial(D^2 \times \bar{D_\gamma}) $, we have 
    \[
        K(X \times S^1 \times (\bar{D_\gamma}, \gamma)) \oplus K(X \times (D^2, S^1) \times \gamma) \cong K(X \times (T, S^1 \times \gamma)). 
    \]
    The element corresponding to $ (u, v) \in K(X \times S^1 \times (\bar{D_\gamma}, \gamma)) \oplus K(X \times (D^2, S^1) \times \gamma) $ is $ [\proj_X^* E, \proj_X^* E, \{\hat{H}_x(z) - \lambda\}_{x, z, \lambda}] \in K(X \times (T, S^1 \times \gamma)) $. 
    Therefore, 
    \begin{align*}
        & \int_X \left(\int_{S^1} \int_{\bar{D_\gamma}} ch(u) + \int_{\gamma} \int_{D^2} ch(v)\right) \\
        &= \int_X \int_T ch([\proj_X^* E, \proj_X^* E, \{\hat{H}_x(z) - \lambda\}_{x, z, \lambda}]) \\
        &= \int_X \int_T ch([\proj_X^* E] - [\proj_X^* E]) = 0,
    \end{align*}
    and $ \mathrm{ind}^b(E, H) = \mathrm{ind}^e(E, H) $ holds. 
    We use the following commutative diagram at the second equality:
    \[
        \begin{tikzcd}
            K(X \times (T, S^1 \times \gamma)) \arrow{r}{ch} \arrow{d} & H^{2*}(X \times (T, S^1 \times \gamma); \Rbb) \arrow{d} \\
            K(X \times T) \arrow{r}{ch} & H^{2*}(X \times T; \Rbb)
        \end{tikzcd}
    \]
\end{proof}

\bibliographystyle{alpha}
\bibliography{bott-periodicity}

\end{document}